\documentclass{amsart}

\usepackage{amsmath,amsfonts,amssymb,amsthm}
\usepackage{mathtools}
\usepackage{mathrsfs}
\usepackage{multirow}
\usepackage{multicol}
\usepackage[inline]{enumitem}
\usepackage{todonotes}
\usepackage{cleveref}

\makeatletter
\newcommand{\inlineitem}[1][]{%
	\ifnum\enit@type=\tw@
	{\descriptionlabel{#1}}
	\hspace{\labelsep}%
	\else
	\ifnum\enit@type=\z@
	\refstepcounter{\@listctr}\fi
	\quad\@itemlabel\hspace{\labelsep}%
	\fi}
\makeatother

\theoremstyle{plain}
\newtheorem{theorem}{Theorem}[section]
\newtheorem{corollary}[theorem]{Corollary}
\newtheorem{proposition}[theorem]{Proposition}
\newtheorem{lemma}[theorem]{Lemma}
\theoremstyle{definition}
\newtheorem{definition}[theorem]{Definition}

\newtheorem{remark}[theorem]{Remark}
\newtheorem{example}[theorem]{Example}

\newcommand{\cT}{\mathcal T}

\newcommand{\bx}{\mathbf x}
\newcommand{\by}{\mathbf y}

\newcommand{\bu}{\mathbf u}

\newcommand{\be}{\mathbf e}

\newcommand{\ba}{\mathbf a}

\newcommand{\bv}{\mathbf v}
\newcommand{\RR}{\mathbb{R}}

\title[ ]{Strictly monotone sequences of lower and upper bounds on Perron values and their combinatorial applications}
\date{\today}
\author{Sooyeong Kim}
\address{Department of Energy, Systems, Territory and Construction Engineering, Universit\'{a} di Pisa, Pisa,
Italy}
\email{kswim2502@gmail.com}

\author{Minho Song}
\address{Applied Algebra and Optimization Research Center, Sungkyunkwan University, Suwon,
  South Korea}
\email{smh3227@skku.edu}

\begin{document}
\begin{abstract}
	In this paper, we present monotone sequences of lower and upper bounds on the Perron value of a nonngeative matrix, and we study their strict monotonicity. Using those sequences, we provide two combinatorial applications. One is to improve bounds on Perron values of rooted trees in combinatorial settings, in order to find characteristic sets of trees. The other is to generate log-concave and log-convex sequences through the monotone sequences.
\end{abstract}


\maketitle
\tableofcontents

\section{Introduction and preliminaries}

The \textit{Perron value} $\rho(A)$ of a square nonnegative matrix $A$, which is the spectral radius of $A$, together with a \textit{Perron vector}, which is an eigenvector of $A$ associated with $\rho(A)$, has been exploited and played important roles in many applications \cite{berman1994nonnegative, horn2012matrix,seneta2006non}. In particular, iterative analysis has contributed to approximating Perron values, how fast iterative methods converge, and so on \cite{sewell2014computational, hogben2013handbook, varga1962iterative}. In this article, rather than dealing with such generic questions in numerical analysis for Perron values, we concentrate our attention on sequences of lower and upper bounds on Perron values that are induced from particular iterative methods, which are presented in this section; and we investigate under what circumstances those sequences are strictly monotone. Using the sequences, we improve bounds on Perron values of bottleneck matrices for trees that are presented in \cite{andrade2017combinatorial, molitierno2018tight}. Furthermore, we show that the sequences with some extra conditions generate log-concave and log-convex sequences, so this can be used as a tool to see if some sequence is log-concave or log-convex.

To elaborate our aim of this article, we begin with some notation and terminologies, and then we present the sequences of lower and upper bounds on Perron values. Any bold-faced letter denotes a column vector, and all matrices are assumed to be real and square throughout this paper. A matrix is \textit{nonnegative} (resp. \textit{positive}) if all the entries are nonnegative (resp. positive). Analogous definitions for a nonnegative vector and a positive vector follow. Let $\RR^n_+$ be the set of all nonnegative vectors in $\RR^n$, and $\RR^n_{++}$ be the set of all positive vectors in $\RR^n$. We denote by $\mathbf{1}_k$ (resp. \( \mathbf{0}_k \)) the all ones vector (resp. the zero vector) of size $k$, and by $J_{p,q}$ (resp. \( \mathbf{O}_{p,q} \)) the all ones matrix (resp. the zero matrix) of size $p\times q$. If $k=p=q$, we write $J_{p,q}$ and $\mathbf{O}_{p,q}$ as $J_k$ and $\mathbf{O}_{k}$. The subscripts $k$ and a pair of $p$ and $q$ are omitted if their sizes are clear from the context. We denote by $\be_i$ the column vector whose component in $i^\text{th}$ position is $1$ and zeros elsewhere. For a vector \( \bx \), \( (\bx)_i \) denotes the \( i^{th} \) component of \( \bx \). A matrix $A$ is said to be \textit{reducible} if there exists a permutation matrix $P$ such that $PAP^T$ is a block upper triangular matrix. If $A$ is not reducible, then we say that $A$ is \textit{irreducible}. We say that a nonnegative matrix $A$ is \textit{primitive} if there exists a positive integer $N$ such that $A^N$ is positive. A symmetric matrix $A$ is said to be \textit{positive definite} (resp. \textit{positive semidefinite}) if all eigenvalues of $A$ are positive (resp. nonnegative).

We state two well-known results for bounds on Perron values. In particular, we adapt the Rayleigh--Ritz theorem as per our purpose. 

\begin{theorem}[The Collatz--Wielandt formula \cite{hogben2013handbook}]\label{CollatzWielandt}
	Let $A$ be an $n\times n$ irreducible nonnegative matrix. Then,
	\begin{align*}
		\rho(A)=\max_{\bx\in \mathbb{R}_{+}^n\backslash\{\mathbf{0}\}}\min_{\{i | (\bx)_i>0\}}\frac{\left(A\bx\right)_i}{\left(\bx\right)_i}=\min_{\by\in\mathbb{R}_{++}^n}\max_{i}\frac{\left(A\by\right)_i}{\left(\by\right)_i},
	\end{align*}
	This implies that for $\bx\in \mathbb{R}_{+}^n\backslash\{\mathbf{0}\}$ and $\by\in \mathbb{R}_{++}^n$, 
	\begin{align*}
		\min_{i}\frac{\left(A\bx\right)_i}{\left(\bx\right)_i}\leq \rho(A)\leq \max_{i}\frac{\left(A\by\right)_i}{\left(\by\right)_i}.
	\end{align*}
\end{theorem}

\begin{theorem}[The Rayleigh--Ritz theorem \cite{hogben2013handbook}]\label{thm:min-max}
	Let \( A \) be an \( n\times n\) nonnegative symmetric matrix. Then,
	\begin{align*} 
	 \rho(A)=\max_{\bx\in\RR^n\backslash\{\mathbf{0}\}} \frac{\bx^TA\bx}{\bx^T\bx}.
	\end{align*}
\end{theorem}

Now we introduce sequences of lower and uppper bounds on Perron values that we shall deal with in this article. Let $A$ be an $n\times n$ nonnegative matrix $A$ with $A\neq \mathbf{O}$. For integer $k\geq 1$, if $A$ is irreducible and $\bx\in\mathbb{R}_{++}^n$, we define $a_k(A,\bx):=\max_i\frac{(A^k\bx)_i}{(A^{k-1}\bx)_i}$ and $b_k(A,\bx):=\min_i\frac{(A^k\bx)_i}{(A^{k-1}\bx)_i}$; and if $A$ is positive semidefinite and $\bx\in\mathbb{R}_{+}^n$, then $c_k(A,\bx):=\frac{\bx^TA^k\bx}{\bx^TA^{k-1}\bx}$. Note that the conditions of \( A \) and \( \bx \) are different when defining \( a_k(A,\bx) \), \( b_k(A,\bx) \) and when defining \( c_k(A,\bx) \).

By \Cref{CollatzWielandt}, $(a_k(A,\bx))_{k\geq 1}$ and $(b_k(A,\bx))_{k\geq 1}$ are sequences of upper and lower bounds on $\rho(A)$, respectively; and we can find from \Cref{thm:min-max} that taking $\by=A^{\frac{k-1}{2}}\bx$ in $c_k(A,\bx)$, $(c_k(A,\bx))_{k\geq 1}$ is a sequence of lower bounds on $\rho(A)$. In addition to those sequences, we refer the reader to \cite{szyld1992sequence,tacscci1998sequence} for another sequences of lower and upper bounds on Perron values.

We review some known results regarding monotonicity and convergence for the three sequences.

\begin{theorem}\cite{varga1962iterative}\label{thm:known}
	Let $A$ be an $n\times n$ irreducible nonnegative matrix, and $\bx\in\mathbb{R}_{++}^n$. Then,
	$$b_1(A,\bx)\leq b_2(A,\bx)\leq\cdots\leq\rho(A)\leq\cdots\leq a_2(A,\bx)\leq a_1(A,\bx).$$
\end{theorem}

\begin{remark}\label{Remark:convergence}
	Let $A$ be irreducible and nonnegative. The sequences in Theorem \ref{thm:known} are not necessarily convergent to the Perron value. A sufficient condition for the convergence of \( a_k(A,\bx) \) and \( b_k(A,\bx) \) is that \( A \) is primitive. By the Perron--Frobenius theorem, if $A$ is primitive, then $\rho(A)$ is greater in absolute value than all other eigenvalues of $A$. Hence it follows from \cite[Theorem 3.5.1]{sewell2014computational} that $(a_k(A,\bx))_{k\geq 1}$ and $(b_k(A,\bx))_{k\geq 1}$ converge to the Perron value as $k\rightarrow\infty$.
\end{remark}

\begin{remark}\label{Remark:convergence2}
	Let $A$ be nonnegative and positive semidefinite. From the power method, one can find that $(c_k(A,\bx))_{k\geq 1}$ converges to $\rho(A)$ as $k\rightarrow\infty$.
\end{remark}

Here we describe the aim of this paper with its motivations. Our approach in \Cref{Sec2:sequences} is informed by the following motivation: in \cite{andrade2017combinatorial}, the so-called \textit{combinatorial Perron value}, which is a lower bound on ``the Perron value of a rooted (unweighted) tree'', may be used to estimate ``characteristic sets'' of trees, which will be explained in \Cref{subsec 1.1}. Not only the combinatorial Perron value, but also other bounds may be used for the estimation, and even sharper bounds improve the accuracy of where the characteristic set is. This leads in \Cref{Sec2:sequences} to explore the strict monotonicity of $(a_k(A,\bx))_{k\geq 1}$, $(b_k(A,\bx))_{k\geq 1}$, and $(c_k(A,\bx))_{k\geq 1}$ that produce sharper bounds than the combinatorial Perron value and other bounds in \cite{andrade2017combinatorial}, which will be shown in Subsection \ref{Subsec:3.1}. Furthermore, those strictly monotone sequences may be used for solving problems concerning bounds on Perron values of nonnegative matrices (especially, combinatorial matrices), only with small powers of the matrices. In Subsection \ref{Subsec:3.1}, we improve the bound in \cite{molitierno2018tight} in order to show its capability.

The other motivation for \Cref{Sec2:sequences} is that the monotonicity of $(a_k(A,\bx))_{k\geq 1}$ and $(b_k(A,\bx))_{k\geq 1}$ with some additional conditions and the monotonicity of $(c_k(A,\bx))_{k\geq 1}$ enable us to obtain log-concave and log-convex sequences, which will be elaborated in Subsection \ref{subsec 1.2}. This can be a tool of proving if a given sequence is log-concave or log-convex, by checking if the sequnce corresponds to one of the three sequences. In \Cref{Sec2:sequences}, we study under what circumstance $(a_k(A,\bx))_{k\geq 1}$ and $(b_k(A,\bx))_{k\geq 1}$ generate log-concave and log-convex sequences. Using our findings, we present combinatorial sequences in Subsection \ref{Subsec:3.2}.

The structure of this paper is as follows. Subsections \ref{subsec 1.1} and \ref{subsec 1.2} contain introduction and necessary background for combinatorial applications in Subsections \ref{Subsec:3.1} and \ref{Subsec:3.2}, respectively. \Cref{Sec2:sequences} provides the results stated above, which are used in Subsections \ref{Subsec:3.1} and \ref{Subsec:3.2}. For the readability of this article, we put in \Cref{appendix} parts of proofs for some results in Subsection \ref{Subsec:3.1} that contain tedious calculations and basic techniques.


\subsection{Combinatorial application I}\label{subsec 1.1}

In this subsection, we aim to understand why we shall study the strict monotonicity of $(a_k(A,\bx))_{k\geq 1}$, $(b_k(A,\bx))_{k\geq 1}$, and $(c_k(A,\bx))_{k\geq 1}$. We assume familiarity with basic material on graph theory. We refer the reader to \cite{chartrand1996graphs} for necessary background. All graphs are assumed to be simple and undirected.

Given a weighted, connected graph $G$ on vertices $1,\dots, n$, the \textit{Laplacian matrix} of $G$ is the $n\times n$ matrix given by $L(G)=[l_{i,j}]$, where $l_{i,j}$ is the weight on edge joining $i$ and $j$ if $i$ and $j$ are adjacent, $l_{i,i}$ is the degree of vertex $i$, and $l_{i,j}=0$ for the remaining entries. The \textit{algebraic connectivity} $a(G)$ of $G$ is the second smallest eigenvalue of $L(G)$, and its corresponding vector is called a \textit{Fiedler vector} of $G$. As the name suggests, this parameter is related to other parameters in terms of connectivity of graphs \cite{fiedler1973algebraic}. For a vertex $v$ of $G$, the \textit{bottleneck matrix $M$ at $v$} is the inverse of the matrix obtained from $L(G)$ by removing the row and column indexed by $v$. If $G$ is an unweighted tree, then the $(i,j)$-entry of $M$ is the number of edges which are simultaneously on the path from $i$ to $v$ and the path from $j$ to $v$; and if $G$ is a weighted tree or other case, we refer the interested reader to \cite{kirkland1996characteristic, kirkland1997distances} for the combinatorial interpretation of the entries of $M$. Suppose that $C_1,\dots,C_k$ are the components of the graph obtained from $G$ by removing $v$ and all incident edges, for some $k\geq 1$ (if $k\geq 2$, then $v$ is called a \textit{cut-vertex}). For $i=1,\dots,n$, we refer to the inverse of the principal submatrix of $L(G)$ corresponding to the vertices of $C_i$, as the \textit{bottleneck matrix for $C_i$}. It is known in \cite{fallat1998extremizing} that $M$ can be expressed as a block diagonal matrix in which the main diagonal blocks consist of the bottleneck matrices for $C_1,\dots,C_k$, which are symmetric and positive matrices. The \textit{Perron value of $C_i$} is defined as the Perron value of the bottleneck matrix for $C_i$. Then, the Peron value of $M$ is determined by the maximum among Perron values of $C_1,\dots,C_k$. We say that $C_i$ is a \textit{Perron component at $v$} if its Perron value is the maximum among Perron values of $C_1,\dots,C_k$. In the context of trees, regarding bottleneck matrices and their Perron values, the word ``component'' is conventionally replaced by ``branch'' so that the related terminologies above will be adapted appropriately; for instance, Perron components at $v$ in a tree are referred to as \textit{Perron branches} at $v$.

As described in the earlier part of this subsection, we now consider unweighted trees instead of weighted and connected graphs. (Some results in terms of weighted and connected graphs will be provided in \Cref{Subsec:3.1}.) In order to understand the characteristic set of a tree, we shall elaborate two characterizations of trees according to Fiedler vectors, and according to Perron branches at particular vertices. (Generalized characterizations for weighted and connected graphs can be found in \cite{fiedler1975property,kirkland1998perron,molitierno2016applications}.)

Let $\cT$ be a tree, and let $\bx$ be its Fiedler vector. It appears in \cite{fiedler1975property} that exactly one of the following cases occurs:
\begin{enumerate}[label=(\roman*)]
	\item\label{typeII} No entry of $\bx$ is zero. Then, there exist unique vertices $i$ and $j$ in $\cT$ such that $i$ and $j$ are adjacent with $x_i>0$ and $x_j<0$. Further, the entries of $\bx$ corresponding to vertices along any path in $\cT$ which starts at $i$ and does not contain $j$ are increasing, while the entries of $\bx$ corresponding to vertices along any path in $\cT$ which starts at $j$ and does not contain $i$ are decreasing. 
	\item\label{typeI} There is a zero entry in $\bx$. For this case, the subgraph induced by the set of vertices corresponding to $0$'s in $\bx$ is connected. Moreover, there is a unique vertex $i$ such that $x_i=0$ and $i$ is adjacent to at least one vertex $j$ with $x_j\neq 0$. The entries of $\bx$ corresponding to vertices along any path in $\cT$ which starts at $i$ are either increasing, decreasing, or identically $0$.
\end{enumerate}
Trees corresponding to \ref{typeI} are said to be \textit{Type I} and the vertex $i$ described in \ref{typeI} is called the \textit{characteristic vertex}; and trees corresponding to \ref{typeII} are said to be \textit{Type II} and the vertices $i$ and $j$ described in \ref{typeII} are called the \textit{characteristic vertices}. We say that the \textit{characteristic set} of a tree is the set of its characteristic vertices. As shown in \cite{merris1987characteristic}, the characteristic set of a tree is independent of the choice of a Fiedler vector. In \cite{andrade2017combinatorial,kirkland1998perron}, the authors regard the characteristic set as a notion of ``middle'' of a tree in the rough sense that the farther away a vertex is from the characteristic set, the larger its corresponding entry in a Fiedler vector is in absolute value. Indeed, the authors of \cite{abreu2017characteristic} studied characteristic set with other notions of middle of a tree---the distances between a centroid of a tree and its characteristic vertices, and between a centre, as a standard notion in graph theory, and its characteristic vertices; so they show that ratios of those maximum distances taken over all trees on $n$ vertices to $n$ are convergent as $n\rightarrow\infty$. For instance, the maximum distance between centroids and characteristic vertices taken over all trees on $n$ vertices asymptotically equals around $0.1129n$---that is, one may expect that the distance between a centroid of a tree on $n$ vertices and its characteristic vertex is less than $0.1129n$.

In order to find the characteristic set of a tree, one may use a Fielder vector by observing its sign patterns. As an alternative, one may use the following characterization in \cite{kirkland1996characteristic} that describes a connection between the algebraic connectivity and Perron branches at characteristic vertices. A tree $\cT$ is Type I if and only if there exists a unique vertex $v$ in $\cT$ such that there are two or more Perron branches $B_1,\dots,B_k$ at $v$ for some $k\geq 2$. For this case, 
\begin{align*}
	a(G)=\frac{1}{\rho(M_i)}
\end{align*}
where $M_i$ is the bottleneck matrix for $B_i$ for $i=1,\dots,k$. A tree $\cT$ is Type II if and only if there exist unique adjacent vertices $i$ and $j$ in $\cT$ such that the branch at $j$ containing $i$ is the unique Perron branch at $j$, and the branch at $i$ containing $j$ is the unique Perron branch at $i$. In this case, there exists $0< \gamma <1$ such that 
\begin{align*}
	a(G)=\frac{1}{\rho(M_1-\gamma J)}=\frac{1}{\rho(M_2-(1-\gamma) J)},
\end{align*}
where $M_1$ (resp. $M_2$) is the bottleneck matrix for the branch at $j$ containing $i$ (resp. at $i$ containing $j$). As suggested in \cite{abreu2017characteristic}, we may estimate characteristic sets of trees through bounds on Perron values of branches---for that, combinatorial Perron value was introduced in that paper. That is, if Perron values of branches at some vertex in a tree are understood well, then one may decide whether it belongs to the characteristic set from the characterization; further, if the algebraic connectivity is also known, then one may find at which vertices Perron branches have their Perron values close to the reciprocal of the algebraic connectivity. 

When it comes to finding the characteristic set through Perron branches and their Perron values, we need to understand Perron values of branches at some vertex $v$. Such branches may be identified as rooted trees, by considering the vertex of a branch adjacent to $v$ as the root. Henceforth, we shall focus on rooted trees instead of branches at some vertex in a tree. Given a rooted tree with root $x$, let us consider the tree $\cT$ formed by adding a new pendent vertex $v$ to $x$. We shall define the \textit{bottleneck matrix $M$ of the rooted tree} as the bottleneck matrix at $v$ in $\cT$, see Figure \ref{Fig:illustration} for an example. Since the rooted tree is unweighted, the $(i,j)$-entry of $M$ is the number of vertices (not edges) which are simultaneously on the path from $i$ to $x$ and on the path from $j$ to $x$. We also define the \textit{Perron value of the rooted tree} to be the Perron value of the bottleneck matrix of the rooted tree. This convention also appears in \cite{andrade2017combinatorial,ciardo2021perron}.
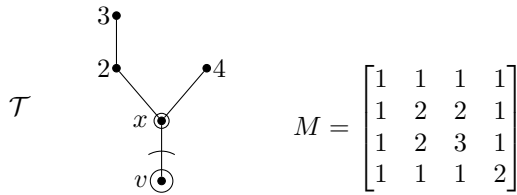
\begin{figure}
	\begin{center}
			\begin{tikzpicture}
				\tikzset{enclosed/.style={draw, circle, inner sep=0pt, minimum size=.10cm, fill=black}}
				
				\node[enclosed, label={left, yshift=0cm: $v$}] (v_0) at (0,1) {};
				\node[enclosed, label={left, yshift=0cm: $x$}] (v_1) at (0,1.8) {};
				\node[enclosed, label={left, xshift=0.1cm: $2$}] (v_2) at (-.6,2.5) {};
				\node[enclosed, label={right, xshift=-.1cm: $4$}] (v_3) at (.6,2.5) {};
				\node[enclosed, label={left, xshift=.1cm: $3$}] (v_4) at (-.6,3.2) {};
				\node[label={left: $\cT$}] (T) at (-1.5,2) {};
				\node[label={right: \( M=\begin{bmatrix}
					1 & 1 & 1 & 1 \\
					1 & 2 & 2 & 1\\
					1 & 2 & 3 & 1 \\
					1 & 1 & 1 & 2 
				\end{bmatrix}\)}] (N) at (1.5,1.7) {};

				\draw (v_0) -- (v_1);
				\draw (v_1) -- (v_2);
				\draw (v_1) -- (v_3);
				\draw (v_2) -- (v_4);

				\path[draw=black] (v_0) circle[radius=0.15];
				\path[draw=black] (0.175,1.35) arc (60:120:0.35);
				\path[draw=black] (v_1) circle[radius=0.1];
				\end{tikzpicture}
	\end{center}
	\caption{An illustration of the bottleneck matrix of a rooted tree. The matrix $M$ is the bottleneck matrix for the branch at $v$ in $\cT$, and $M$ is also the bottleneck matrix of a rooted tree with vertex set $\{x,2,3,4\}$ and root $x$.}\label{Fig:illustration}
\end{figure}

\subsection{Combinatorial application II}\label{subsec 1.2}

A sequence $(z_n)_{n\geq 1}$ is \textit{log-concave} (resp. \textit{log-convex}) if $z_{n-1}z_{n+1}\leq z_n^2$ (resp. $z_{n-1}z_{n+1}\geq z_n^2$) for $n\geq 2$; and if the inequality is strict, $(z_n)_{n\geq 1}$ is \textit{strictly log-concave} (resp. \textit{strictly log-convex}). We refer the reader to \cite{brenti1989unimodal,huh2018combinatorial,stanley1989log} for an introduction and applications. The paper \cite{gross2015log} deals with a conjecture arising in topological graph theory that the genus distribution of every graph is log-concave. In \cite{liu2007log}, one can find operations preserving log-convexity and conditions for a sequence to be log-concave concerning recurrence relations. For the log-concavity of symmetric functions, see \cite{sagan1992log}.

We shall provide a systematic way of generating (strictly) log-concave and (strictly) log-convex sequences from the sequences $(a_k(A,\bx))_{k\geq 1}$, $(b_k(A,\bx))_{k\geq 1}$, and $(c_k(A,\bx))_{k\geq 1}$.

\begin{definition}\label{def:log index}
	Let $A$ be an $n\times n$ irreducible nonnegative matrix and $\bx\in\RR_{++}^n$. For $i_0\in\{1,\dots,n\}$, we say that $i_0$ is a \textit{log-concavity (resp. log-convexity) index of $A$ associated with $\bx$}, or equivalently that $A$ has a \textit{log-concavity (resp. log-convexity) index $i_0$ associated with $\bx$}, if there exists a positive number $K$ such that for $k\geq K$,
	\begin{align*}
		\frac{(A^k\bx)_{i_0}}{(A^{k-1}\bx)_{i_0}}=\max_{i}\frac{(A^k\bx)_{i}}{(A^{k-1}\bx)_{i}},\; \left(\text{resp.} \frac{(A^k\bx)_{i_0}}{(A^{k-1}\bx)_{i_0}}=\min_{i}\frac{(A^k\bx)_{i}}{(A^{k-1}\bx)_{i}}\right).
	\end{align*}
\end{definition}

From the following proposition, we can see that the names ``log-concavity index'' and ``log-convexity index''  suggest log-concave and log-convex sequences, respectively. 


\begin{proposition}\label{Prop:logconcavity index}
	Let $A$ be an $n\times n$ irreducible nonnegative matrix and $\bx\in\RR_{++}^n$. Then, the following hold:
	\begin{enumerate}[label=(\roman*)]
		\item Suppose that $i_1$ is a log-concavity index of $A$ associated with $\bx$. Let $g_k=(A^k\bx)_{i_1}$ for $k\geq 0$. Then, there exists some $k_1>0$ such that $(g_k)_{k\geq k_1-1}$ is log-concave. Moreover, if the sequence $(a_k(A,\bx))_{k\geq k_1}$ is strictly decreasing, then $(g_k)_{k\geq k_1-1}$ is strictly log-concave.
  		
		\item Suppose that $i_2$ is a log-convexity index of $A$ associated with $\bx$. Let $h_k=(A^k\bx)_{i_2}$ for $k\geq 0$. Then, there exists some $k_2>0$ such that $(h_k)_{k\geq k_2-1}$ is log-convex. Moreover, if the sequence $(b_k(A,\bx))_{k\geq k_2}$ is strictly increasing, then $(h_k)_{k\geq k_2-1}$ is strictly log-convex.
	\end{enumerate}
\end{proposition}
\begin{proof}
	Suppose that $i_1$ is a log-concavity index of $A$ associated with $\bx$. Then, there exists some $k_1>0$ such that $\frac{g_{k}}{g_{k-1}}=a_k(A,\bx)$ for $k\geq k_1$. By \Cref{thm:known}, $(a_k(A,\bx))_{k\geq 1}$ is decreasing, so $\frac{g_{k+1}}{g_{k}}\leq \frac{g_{k}}{g_{k-1}}$ for $k\geq k_1$. Hence, $(g_k)_{k\geq k_1}$ is log-concave. Moreover, if the sequence $(a_k(A,\bx))_{k\geq k_1}$ is strictly decreasing, then $(g_k)_{k\geq k_1}$ is strictly log-concave. From a similar argument, one can establish the remaining conclusions.
\end{proof}


\begin{proposition}\label{prop:log-concave from c_k}
	Let $A$ be an $n\times n$ nonnegative, positive semidefinite matrix and $\bx\in\RR_{+}^n$. Let $s_k=\bx^TA^k\bx$ for $k\geq 0$. If the sequence $(c_k(A,\bx))_{k\geq 1}$ is (strictly) increasing, then $(s_k)_{k\geq 0}$ is (strictly) log-convex.
\end{proposition}

We remark that given a sequence $(x_k)_{k\geq 1}$, if there exist some nonnegative matrix and nonnegative vector so that one of $(a_k(A,\bx))_{k\geq 1}$, $(b_k(A,\bx))_{k\geq 1}$, and $(c_k(A,\bx))_{k\geq 1}$ generates $(x_k)_{k\geq 1}$ by one of \Cref{Prop:logconcavity index,prop:log-concave from c_k}, then $(x_k)_{k\geq 1}$ is log-concave or log-convex. 

We examine under what circumstances the sufficient conditions of \Cref{Prop:logconcavity index,prop:log-concave from c_k} are satisfied in \Cref{Sec2:sequences}, and so related results are presented in \Cref{thm:logconcave seq,Thm:log-concave from c_k}.

\section{Strictly monotone sequences of lower and upper bounds on Perron values}\label{Sec2:sequences}

Our main goal of this section is to find a condition for an irreducible nonnegative matrix to have a log-concavity or log-convexity index associated with a positive vector (\Cref{prop:log index}), in order to improve \Cref{Prop:logconcavity index}, and to find conditions for sequences $(a_k(A,\bx))_{k\geq 1}$, $(b_k(A,\bx))_{k\geq 1}$ and $(c_k(A,\bx))_{k\geq 1}$ to be strictly monotone. Specifically speaking of the latter, we explore conditions on an $n\times n$ irreducible nonnegative matrix $A$ and $\bx\in\RR_{++}^n$ such that the corresponding sequences $(a_k(A,\bx))_{k\geq 1}$ and $(b_k(A,\bx))_{k\geq 1}$ are strictly monotone (Theorems \ref{thm:UpperMonocity} and \ref{thm:UpperMonocity2}), and we find conditions on an $n\times n$ nonnegative, positive semidefinite matrix $A$ and $\bx\in\RR_{+}^n$ such that the corresponding sequence $(c_k(A,\bx))_{k\geq 1}$ is strictly increasing (Theorem \ref{Cor:LowerMonocity}).

\subsection{Log-concavity and log-convexity indices}

We begin with the following proposition, which is used for obtaining \Cref{lem:index domnates others in two vectors}.

\begin{proposition}\label{prop:index domnates others in two vectors}
	Let $n\geq 2$. Let $\bx=(x_1,\dots,x_n)^T\in\RR_{++}^n$, and $\by=(y_1,\dots,y_n)^T\in\RR^n$. Suppose that either $\by\in\RR_{++}^n$ or $-\by\in\RR_{++}^n$. Then, 
	\begin{enumerate}[label=(\roman*)]
		\item\label{statement1} there exists $i_1\in\{1,\dots,n\}$ such that for all $j_1\in\{1,\dots, n\}$,
		\begin{align*}
		\det\begin{bmatrix}
		x_{i_1} & y_{i_1}\\
		x_{j_1} & y_{j_1}
		\end{bmatrix}\geq 0;
		\end{align*}
		\item\label{statement2} there exists $i_2\in\{1,\dots,n\}$ such that for all $j_2\in\{1,\dots, n\}$,
		\begin{align*}
		\det\begin{bmatrix}
		x_{i_2} & y_{i_2}\\
		x_{j_2} & y_{j_2}
		\end{bmatrix}\leq 0.
		\end{align*}
	\end{enumerate}
\end{proposition}
\begin{proof}
	Suppose that $\by\in\RR_{++}^n$. We shall prove the statement \ref{statement1} by induction on $n$. Clearly, it holds for $n=2$. Let $n\geq 3$. Consider $x_2,\dots,x_n$ and $y_2,\dots,y_n$. By the induction hypothesis, there exists $k_1$ in $\{2,\dots,n\}$ such that $x_{k_1}y_{j_1}-y_{k_1}x_{j_1}\geq 0$ for all $j_1\in\{2,\dots, n\}$.
	
	Suppose that $x_1y_{k_1}-y_1x_{k_1}\geq0$. Since $x_1y_{k_1}\geq y_1x_{k_1}>0$ and $x_{k_1}y_{j_1}\geq y_{k_1}x_{j_1}>0$, we have $x_1y_{k_1}x_{k_1}y_{j_1}\geq y_1x_{k_1}y_{k_1}x_{j_1}$, and so $x_1y_{j_1}\geq y_1x_{j_1}$. Hence, $\det\begin{bmatrix}
	x_{1} & y_{1}\\
	x_{j} & y_{j}
	\end{bmatrix}\geq 0$ for $j=2,\dots,n$, and so $1$ is our desired index in \ref{statement1}. If $x_1y_{k_1}-y_1x_{k_1}\leq 0$, then $\det\begin{bmatrix}
	x_{k_1} & y_{k_1}\\
	x_{j} & y_{j}
	\end{bmatrix}\geq 0$ for all $j\in\{1,\dots, n\}$. Hence, by induction, the statement \ref{statement1} holds for $\by>0$. 
		
	Assuming $-\by\in\RR_{++}^n$, an analogous argument establishes \ref{statement1}.
	
	Note that for a square matrix, a change of the sign of a column switches the sign of the determinant of the matrix. 
	Therefore, by changing the sign of the vector \( \by \) above, the remaining conclusion follows.
\end{proof}

\begin{lemma}\label{lem:index domnates others in two vectors}
	Let $n\geq 2$. Let $\bx=(x_1,\dots,x_n)^T\in\RR_{++}^n$, and $\by=(y_1,\dots,y_n)^T\in\RR^n$. Then, the statements \ref{statement1} and \ref{statement2} in Proposition \ref{prop:index domnates others in two vectors} hold.
\end{lemma}
\begin{proof}
	Let $R_1=\{1\leq i\leq n\,|\,y_i<0\}$, $R_2=\{1\leq i\leq n\,|\,y_i=0\}$, and $R_3=\{1\leq i\leq n\,|\,y_i>0\}$. By Proposition \ref{prop:index domnates others in two vectors}, there exist $k_1\in R_1$ and $k_3\in R_3$ such that $\mathrm{det}\begin{bmatrix}
	x_{k_1} & y_{k_1}\\
	x_{j_1} & y_{j_1}
	\end{bmatrix}\geq 0$ for $j_1\in R_1$ and $\mathrm{det}\begin{bmatrix}
	x_{k_3} & y_{k_3}\\
	x_{j_3} & y_{j_3}
	\end{bmatrix}\leq 0$ for $j_3\in R_3$. We can readily see that $\mathrm{det}\begin{bmatrix}
	x_{k_1} & y_{k_1}\\
	x_{j} & y_{j}
	\end{bmatrix}>0$ for $j\in R_2\cup R_3$; and $\mathrm{det}\begin{bmatrix}
	x_{k_3} & y_{k_3}\\
	x_{j} & y_{j}
	\end{bmatrix}<0$ for $j\in R_1\cup R_2$. Therefore, from the indices $k_1$ and $k_3$, our desired conclusion follows.
\end{proof}

Here is an interim result to deduce main results of this section (Theorems \ref{thm:logconcave seq} and \ref{thm:UpperMonocity2}).

\begin{proposition}\label{prop:log index}
	Let $A$ be an $n\times n$ irreducible, nonnegative, positive semidefinite matrix, and let $\bx\in\RR_{++}^n$. Then, there exists a log-concavity (resp. log-convexity) index of $A$ associated with $\bx$.
\end{proposition}

\begin{proof}
	Let $F(p,q;k)=\left(\be_p^T A^k\bx\right)\left(\be_q^T A^{k-1}\bx\right)-\left(\be_q^T A^k\bx\right)\left(\be_p^T A^{k-1}\bx\right)$ for $k\geq 1$ and $p,q\in\{1,\dots,n\}$.  Then, $F(p,q;k)\ge0$ if and only if $\frac{(A^k\bx)_{p}}{(A^{k-1}\bx)_{p}}\ge\frac{(A^k\bx)_{q}}{(A^{k-1}\bx)_{q}}$. In order to show the existence of a log-concavity index of $A$ associated with $\bx$, we shall prove the following claim.
	\begin{center}
		\begin{enumerate}[label=(C\arabic*)]
		\item\label{claim 1} There exist $\hat{p}\in\{1,\dots,n\}$ and $K\geq 1$ such that for each $q\in\{1,\dots,n\}$, $F(\hat{p},q;k)\ge 0$ for all $k\geq K$. 
		\end{enumerate}
	\end{center}
	
	Let $\mu_1,\dots,\mu_\ell$ be the distinct eigenvalues of $A$ with $\mu_1>\dots>\mu_\ell\geq 0$ for some $\ell\geq 2$. Let $E_i$ be the orthogonal projection matrix onto the eigenspace of $A$ associated with $\mu_i$ for $i=1,\dots,\ell$. From the spectral decomposition, we have
	\begin{align*}
	A^k=\sum_{i=1}^{\ell}\mu_i^kE_i.
	\end{align*}
	Let $\by_i=E_i\bx$ for $i=1,\dots,n$, and let \( y_{s,t}=(\by_t)_s \). We can find that
	\begin{align}\nonumber
	F(p,q;k)=&\left(\sum_{i=1}^{\ell}\mu_i^ky_{p,i}\right)\left(\sum_{i=1}^{\ell}\mu_i^{k-1}y_{q,i}\right)-\left(\sum_{i=1}^{\ell}\mu_i^ky_{q,i}\right)\left(\sum_{i=1}^{\ell}\mu_i^{k-1}y_{p,i}\right)\\\nonumber
	=&\sum_{1\leq i<j\leq \ell}(\mu_i\mu_j)^{k-1}(\mu_i-\mu_j)y_{p,i}y_{q,j}-\sum_{1\leq i<j\leq \ell}(\mu_i\mu_j)^{k-1}(\mu_i-\mu_j)y_{q,i}y_{p,j}\\\label{det condition}
	=&\sum_{1\leq i<j\leq \ell}(\mu_i\mu_j)^{k-1}(\mu_i-\mu_j)\det\begin{bmatrix}
		y_{p,i} & y_{p,j}\\
		y_{q,i} & y_{q,j}
	\end{bmatrix}.
	\end{align}

	Before we consider the claim \ref{claim 1}, we shall find the dominant term of $F(p,q;k)$ for sufficiently large \( k \), which determines the sign of $F(p,q;k)$. For brevity, we write \( p\succ q \) (resp. $p\succeq q$) if \( F(p,q;k)>0 \) (resp. $F(p,q;k)\geq 0$) for sufficiently large \( k \). Let $p,q\in\{1,\dots,n\}$. Suppose that $F(p,q;k)\neq 0$ for some $k\geq 1$. Then,
	$$m_0=\min\left\{m=2,\dots,n\,\middle|\,\det\begin{bmatrix}
		y_{p,1} & y_{p,m}\\
		y_{q,1} & y_{q,m}
	\end{bmatrix}\neq 0\right\}$$ is well-defined. Suppose that $\mu_{m_1}\mu_{m_2}\geq \mu_1\mu_{m_0}$ for some $1\leq m_1<m_2\leq n$ with $(m_1,m_2)\neq (1,m_0)$. Then, $m_1$ and $m_2$ must be between $1$ and ${m_0}$. Since $\det\begin{bmatrix}
		y_{p,1} & y_{p,m_1}\\
		y_{q,1} & y_{q,m_1}
	\end{bmatrix}=\det\begin{bmatrix}
		y_{p,1} & y_{p,m_2}\\
		y_{q,1} & y_{q,m_2}
	\end{bmatrix}=0$, we have $\det\begin{bmatrix}
		y_{p,m_1} & y_{p,m_2}\\
		y_{q,m_1} & y_{q,m_2}
	\end{bmatrix}=0$. Hence, the dominant term of $F(p,q;k)$ is 
	\begin{equation}\label{eq:dominant term of F}
		(\mu_1-\mu_{m_0})\det\begin{bmatrix}
			y_{p,1} & y_{p,{m_0}}\\
			y_{q,1} & y_{q,{m_0}}
		\end{bmatrix}(\mu_1\mu_{m_0})^{k-1}.		
	\end{equation}
	Thus, if $\det\begin{bmatrix}
		y_{p,1} & y_{p,{m_0}}\\
		y_{q,1} & y_{q,{m_0}}
	\end{bmatrix}>0$, then \( p\succ q \). Note that $\det\begin{bmatrix}
		y_{p,1} & y_{p,{m_0}}\\
		y_{q,1} & y_{q,{m_0}}
	\end{bmatrix}=0$ does not imply that $F(p,q,k)=0$ for sufficiently large $k$.

	In order to establish the claim \ref{claim 1}, it suffices to show the following claim.
	\begin{center}
		\begin{enumerate}[label=(C2)]
			\item\label{claim 2} There exists $\hat{p}\in\{1,\dots,n\}$ such that $\hat{p}\succeq q$ for all $q\in\{1,\dots,n\}$.
		\end{enumerate} 	
	\end{center}
	 We note that if $A^{l_1}\bx$ is a Perron vector of $A$ for some $l_1\geq 0$, so is $A^{l_2}\bx$ for $l_2\geq l_1$. If there exists an integer $K\geq 1$ such that $F(p,q;k)=0$ for all $p,q\in\{1,\dots,n\}$ and $k\geq K$, \textit{i.e.}, $A^{k-1}\bx$ is a Perron vector of $A$ for each $k\geq K$, then we have a log-concavity index. We suppose that for each $k\geq 1$, $F(p,q;k)\neq 0$ for some $p,q\in\{1,\dots,n\}$, that is, $A^{k-1}\bx$ is not a Perron vector for each $k\geq 1$. Then, we may choose $$j_0=\min\left\{j=2,\dots,n\,\middle|\,\det\begin{bmatrix}
		y_{p,1} & y_{p,j}\\
		y_{q,1} & y_{q,j}
	\end{bmatrix}\neq 0\;\text{for some}\; 1\leq p,q\leq n\right\}.$$
	By Lemma \ref{lem:index domnates others in two vectors}, there exists $p_0\in\{1,\dots, n\}$ such that 
	$$\det\begin{bmatrix}
		y_{p_0,1} & y_{p_0,j_0}\\
		y_{q,1} & y_{q,j_0}
	\end{bmatrix}\geq 0$$ 
	for $q\in\{1,\dots, n\}$. Let $$X_1=\left\{q=1,\dots,n\,\middle|\,\det\begin{bmatrix}
		y_{p_0,1} & y_{p_0,j_0}\\
		y_{q,1} & y_{q,j_0}
	\end{bmatrix}=0\right\}.$$
	For each $q\in\{1,\dots, n\}\backslash X_1$, the dominant term of $F(p_0,q,k)$ is positive, and thus \( p_0\succ q \). If $X_1=\{p_0\}$, then $p_0$ is the log-concavity index.
	
	Suppose $|X_1|>1$. Then, to complete the proof, we need to find some $p_1\in X_1$ such that $p_1\succeq q$ for $q\in X_1$, which implies $p_1\succeq q$ for $q\in \{1,\dots,n\}$. Note that considering how $j_0$ is defined, we have $\det\begin{bmatrix}
		y_{p,1} & y_{p,j}\\
		y_{q,1} & y_{q,j}
	\end{bmatrix}=0$ for $p,q\in X_1$ and \( 2\leq  j\le j_0 \). If $\det\begin{bmatrix}
		y_{p,1} & y_{p,j}\\
		y_{q,1} & y_{q,j}
	\end{bmatrix}=0$ for all $p,q\in X_1$ and $j_0+1\leq j\leq n$, then we have $F(p,q;k)= 0$ for all $p,q\in X_1$ and $k\geq 1$, and it follows that all elements in $X_1$ are log-concavity indices. We now suppose that
	$$j_1=\min\left\{j=j_0+1,\dots,n\,\middle|\,\det\begin{bmatrix}
		y_{p,1} & y_{p,j}\\
		y_{q,1} & y_{q,j}
	\end{bmatrix}\neq 0\;\text{for some}\; p,q\in X_1\right\}$$
	is well-defined. Applying Lemma \ref{lem:index domnates others in two vectors}, there exists $p_1\in X_1$ (by abuse of notation) such that 
	$$\det\begin{bmatrix}
		y_{p_1,1} & y_{p_1,j_1}\\
		y_{q,1} & y_{q,j_1}
	\end{bmatrix}\geq 0$$ 
	for $q\in X_1$.  Let 
	$$X_2=\left\{q\in X_1\,\middle|\,\det\begin{bmatrix}
		y_{p_1,1} & y_{p_1,j_1}\\
		y_{q,1} & y_{q,j_1}
	\end{bmatrix}=0\right\}.$$
	If $X_2=\{p_1\}$, then it follows that $p_1$ is the log-concavity index. If $|X_2|>1$, one can apply a similar argument as done for the case $|X_1|>1$. Continuing this process with a finite number of steps, we may conclude that there exists a log-concavity index, as claimed.

	Regarding the existence of a log-convexity index of $A$ associated with $\bx$, applying an analogous argument above, one can prove that there exists an index $\hat{r}$ such that $q \succeq \hat{r}$ for all $q\in\{1,\dots,n\}$.
\end{proof}

\begin{remark}\label{Remark:finding index}
	Continuing Proposition \ref{prop:log index} with the same notation, examining the proof of that proposition, one can find log-concavity and log-convexity indices with a few steps, which will be demonstrated in \Cref{Subsec:3.2}, for the following particular cases: \begin{enumerate*}[label=(\alph*)]
	\item\label{case a} $|X_1|=1$, \item\label{case b} $|X_1|>1$ and $\det\begin{bmatrix}
		y_{p,1} & y_{p,j}\\
		y_{q,1} & y_{q,j}
	\end{bmatrix}=0$ for all $p,q\in X_1$ and $2\leq j\leq n$.
	\end{enumerate*}
	For the former, the element in $X_1$ is a log-concavity index of $A$. For the latter, all elements in $X_1$ are log-concavity indices of $A$. For both cases, if $A^{k-1}\bx$ is not a Perron vector for $k\geq 1$, then the set of log-concavity indices of $A$ is $X_1$. One can apply a similar argument for log-convexity index of $A$.
\end{remark}

\begin{theorem}\label{thm:logconcave seq}
	Let $A$ be an $n\times n$ irreducible, nonnegative, positive semidefinite matrix and $\bx\in\RR_{++}^n$. From Proposition \ref{prop:log index}, we may define $i_1$ and $i_2$ to be a log-concavity index and log-convexity index, respectively. We let $g_k=(A^{k-1}\bx)_{i_1}$ and $h_k=(A^{k-1}\bx)_{i_2}$ for $k\geq 1$. Then, $(g_k)_{k\geq k_1}$ is log-concave for some $k_1\geq 1$, and $(h_k)_{k\geq k_2}$ is log-convex for some $k_2\geq 1$.
\end{theorem}
\begin{proof}
	It follows from \Cref{Prop:logconcavity index}.
\end{proof}

\subsection{Strict monotonicity of sequences $(a_k(A,\bx))_{k\geq 1}$, $(b_k(A,\bx))_{k\geq 1}$, and $(c_k(A,\bx))_{k\geq 1}$}

First, we shall find two conditions for sequences $(a_k(A,\bx))_{k\geq 1}$ and $(b_k(A,\bx))_{k\geq 1}$ to be strictly monotone. We begin with a simple lemma, which will be used in Propositions \ref{prop:inequality} and \ref{prop:inequality2}. 

\begin{lemma}\label{lem:SumOfRatios}
	Let $a_i$ and $b_i$ be positive numbers for $i=1,\dots,n$. Then,
    \begin{enumerate}[label=(\roman*)]
        \item\label{fact1} if $\frac{a_1}{b_1}\geq \frac{a_j}{b_j}$ for $1\leq j\leq n$, then
        $
        \frac{a_1}{b_1}\geq \frac{a_1+\cdots+a_n}{b_1+\cdots+b_n}
        $; and\\
        \item\label{fact2} if $\frac{a_1}{b_1}\leq \frac{a_j}{b_j}$ for $1\leq j\leq n$, then
        $
        \frac{a_1}{b_1}\leq \frac{a_1+\cdots+a_n}{b_1+\cdots+b_n}.
        $
    \end{enumerate}
\end{lemma}


\begin{remark}\label{lem:notperron}
	Let $A$ be an $n\times n$ irreducible nonnegative matrix and $\bx\in\mathbb{R}_{++}^n$. It can be readily seen that $\bx$ is not a Perron vector if and only if there exists $j\in\{1,\dots,n\}$ such that $\frac{\left(A\bx\right)_{j}}{\left(\bx\right)_{j}}\neq \max_{i}\frac{\left(A^{2}\bx\right)_i}{\left(A\bx\right)_i}$.
\end{remark}
\begin{proposition}\label{prop:inequality}
	Let $A$ be an $n\times n$ irreducible nonnegative matrix and $\bx\in\mathbb{R}_{++}^n$. Suppose that $\bx$ is not a Perron vector. 
	Then, we have the following:
	\begin{enumerate}[label=(\roman*)]
		\item\label{result1} If $\max_{i}\frac{\left(A^{2}\bx\right)_i}{\left(A\bx\right)_i}\leq \frac{\left(A\bx\right)_{j}}{\left(\bx\right)_{j}}$ for $j=1,\dots,n$, then $\max_{i}\frac{\left(A^{2}\bx\right)_i}{\left(A\bx\right)_i}< \max_{i}\frac{\left(A\bx\right)_i}{\left(\bx\right)_i}$.
		\item\label{result2} Suppose that there exist $i_0,j_0\in\{1,\dots,n\}$ such that $\frac{\left(A\bx\right)_{j_0}}{\left(\bx\right)_{j_0}}<\max_{i}\frac{\left(A^{2}\bx\right)_i}{\left(A\bx\right)_i}=\frac{\left(A^2\bx\right)_{i_0}}{\left(A\bx\right)_{i_0}}$. Assume $a_{i_0,j_0}\neq 0$. Then, $\max_{i}\frac{\left(A^{2}\bx\right)_i}{\left(A\bx\right)_i}< \max_{i}\frac{\left(A\bx\right)_i}{\left(\bx\right)_i}$.
	\end{enumerate}
\end{proposition}
\begin{proof}
Under the hypothesis of \ref{result1}, since \( \bx \) is not a Perron vector, we have $\max_{i}\frac{\left(A^{2}\bx\right)_i}{\left(A\bx\right)_i}\neq\max_{i}\frac{\left(A\bx\right)_i}{\left(\bx\right)_i}$ and hence \ref{result1} follows.
	
Now we suppose that the hypotheses of \ref{result2} holds. Considering Theorem~\ref{thm:known}, we assume to the contrary that 
$$\max_{i}\frac{\left(A^{2}\bx\right)_i}{\left(A\bx\right)_i}=\max_{i}\frac{\left(A\bx\right)_i}{\left(\bx\right)_i}.$$
Let $i_0$ and $j_0$ be the indices in the hypotheses. Then, 
$$\frac{\left(A^{2}\bx\right)_{i_0}}{\left(A\bx\right)_{i_0}}=\max_{i}\frac{\left(A\bx\right)_i}{\left(\bx\right)_i}\geq \frac{\left(A\bx\right)_j}{\left(\bx\right)_j}$$ 
for all $1\leq j\leq n$. Let $\ba_i^T$ be the $i^\text{th}$ row of the matrix $A$ for $i=1,\dots, n$. For each $j=1,\dots,n$,
\begin{align}\nonumber
	&\frac{\be_{i_0}^TA^2\bx}{\be_{i_0}^TA\bx}-\frac{\be_j^TA\bx}{\be_j^T\bx}\geq 0\\\label{Temp:inequality0}
	\iff & \frac{\be_{i_0}^TA^2\bx}{\be_{i_0}^TA\bx}\be_j^T\bx-\be_j^TA\bx\geq  0\\\nonumber
	\iff & (\be_{i_0}^TA^2\bx)(\be_j^T\bx)-(\be_{i_0}^TA\bx)(\be_j^TA\bx)\geq  0\\\label{Temp:inequality1}
	\iff & (\ba_{i_0}^TA\bx)x_j-(\ba_{i_0}^T\bx)(\ba_j^T\bx)\geq 0\\\label{Temp:inequality}
	\iff & \ba_j^T\bx\leq \frac{\left(\sum_{k=1}^n a_{i_0,k}\ba_k^T\bx\right)x_j}{\ba_{i_0}^T\bx}. \qquad ( \ba_{i_0}^T\bx>0 )
\end{align}

 By the hypothesis, the equality in \eqref{Temp:inequality} does not hold for \( j=j_0 \). Since $a_{i_0,j_0}\neq 0$, taking summation both sides in \eqref{Temp:inequality}, we have
\[ 
	\sum_{l=1}^n a_{i_0,l}\ba_l^T\bx<\sum_{l=1}^n a_{i_0,l}\left(\frac{\left(\sum_{k=1}^n a_{i_0,k}\ba_k^T\bx\right)x_l}{\ba_{i_0}^T\bx}\right).
\]

Now let $\frac{\left(A\bx\right)_{k_0}}{\left(\bx\right)_{k_0}}=\max_{i}\frac{\left(A\bx\right)_i}{\left(\bx\right)_i}$ for some $k_0$. Since \( \max_{i}\frac{\left(A^{2}\bx\right)_i}{\left(A\bx\right)_i}>\frac{\left(A\bx\right)_{j_0}}{\left(\bx\right)_{j_0}} \), we must have $j_0\neq k_0$. Starting from \eqref{Temp:inequality1} for $j=k_0$, we obtain a contradiction from the following argument:
\begin{align}\nonumber
	0\leq&\left(\sum_{l=1}^n a_{i_0,l}\ba_l^T\bx\right)x_{k_0}-\left(\sum_{l=1}^na_{i_0,l}x_l\right)(\ba_{k_0}^T\bx)\\\nonumber
	<&\left(\sum_{l=1}^n a_{i_0,l}\left(\frac{\left(\sum_{k=1}^n a_{i_0,k}\ba_k^T\bx\right)x_l}{\ba_{i_0}^T\bx}\right)\right)x_{k_0}-\left(\sum_{l=1}^na_{i_0,l}x_l\right)(\ba_{k_0}^T\bx)\\\nonumber
	=&\sum_{l=1}^n a_{i_0,l}x_l\left(\frac{\sum_{k=1}^n a_{i_0,k}\ba_k^T\bx}{\sum_{k=1}^na_{i_0,k}x_k}x_{k_0}-\ba_{k_0}^T\bx\right)\\\label{Temp:inequality2}
	\leq&\sum_{l=1}^n a_{i_0,l}x_l\left(\frac{\ba_{k_0}^T\bx}{x_{k_0}}x_{k_0}-\ba_{k_0}^T\bx\right)=0,
\end{align}
The inequality in \eqref{Temp:inequality2} follows from \ref{fact1} of Lemma \ref{lem:SumOfRatios} with \( \frac{\ba_{k_0}^T\bx}{x_{k_0}}\ge \frac{\ba_k^T\bx}{x_k}=\frac{a_{i_0,k}\ba_k^T\bx}{a_{i_0,k}x_k} \) for \( 1\le k\le n \). Therefore, $\max_{i}\frac{\left(A^{2}\bx\right)_i}{\left(A\bx\right)_i} < \max_{i}\frac{\left(A\bx\right)_i}{\left(\bx\right)_i}$. 
\end{proof}

\begin{remark}
	Consider the following irreducible nonnegative matrix $A$:
	\begin{align*}
		A=\begin{bmatrix}
			2 & 2 & 2\\
			3 & 3 & 0\\
			1 & 1 & 1
		\end{bmatrix}.
	\end{align*}
	Let $\bx=\mathbf{1}$. It can be readily checked that $6=\max_i\frac{(A^2\bx)_i}{(A\bx)_i}=\frac{(A^2\bx)_2}{(A\bx)_2}>\frac{(A^2\bx)_1}{(A\bx)_1}=\frac{(A^2\bx)_3}{(A\bx)_3}$; and $6=\frac{(A\bx)_1}{(\bx)_1}=\frac{(A\bx)_2}{(\bx)_2}>\frac{(A\bx)_3}{(\bx)_3}$. Since $a_{2,3}=0$, we can see that the hypothesis of \ref{result2} in Proposition \ref{prop:inequality} is not satisfied. Moreover,
	$\max_{i}\frac{\left(A^{2}\bx\right)_i}{\left(A\bx\right)_i}=\max_{i}\frac{\left(A\bx\right)_i}{\left(\bx\right)_i}$.
\end{remark}

\begin{remark}
	In this remark, we shall provide a concrete example of \textit{non-semipositivity vectors} \cite{tsatsomeros2016geometric} for non-singular $M$-matrices.
	
	Continuing \ref{result2} of Proposition~\ref{prop:inequality} with the same notation, the left side of the inequality in \eqref{Temp:inequality0} can be written as 
	\begin{align*}
		\left(\left(\frac{\be_{i_0}^TA^2\bx}{\be_{i_0}^TA\bx}\right)I-A \right)\bx.
	\end{align*}
	Let $B=\left(\frac{\be_{i_0}^TA^2\bx}{\be_{i_0}^TA\bx}\right)I-A$. Suppose that $\frac{\be_{i_0}^TA^2\bx}{\be_{i_0}^TA\bx}>\rho(A)$. Then, $B$ is a non-singular $M$-matrix, and so $B$ is \textit{semipositive} (See \cite[Chapter 6]{berman1994nonnegative}). Let $K_B=\{ \by\in\mathbb{R}^n_+ | B\by\in\mathbb{R}^n_+\}$. The set $K_B$ is known as the so-called \textit{semipositive cone} of $B$, and $K$ is a proper polyhedral cone in $\mathbb{R}^n$ (see \cite{sivakumar2018semipositive}). Clearly, a Perron vector of $A$ is in $K_B$. Examining the proof of Proposition \ref{prop:inequality}, $B\bx$ contains at least one negative entry. Hence, $\bx\notin K_B$.
\end{remark}

As done in Proposition~\ref{prop:inequality}, one can establish the following with \ref{fact2} of Lemma \ref{lem:SumOfRatios}.

\begin{proposition}\label{prop:inequality2}
	Let $A$ be an $n\times n$ irreducible nonnegative matrix and $\bx\in\mathbb{R}_{++}^n$. Suppose that $\bx$ is not a Perron vector. Then, we have the following:
	\begin{enumerate}[label=(\arabic*)]
		\item If $\min_{i}\frac{\left(A^{2}\bx\right)_i}{\left(A\bx\right)_i}\geq \frac{\left(A\bx\right)_{j}}{\left(\bx\right)_{j}}$ for $j=1,\dots,n$, then $\min_{i}\frac{\left(A^{2}\bx\right)_i}{\left(A\bx\right)_i} > \min_{i}\frac{\left(A\bx\right)_i}{\left(\bx\right)_i}$.
		\item Suppose that there exists indices $i_0$ and $j_0$ such that $\frac{\left(A\bx\right)_{j_0}}{\left(\bx\right)_{j_0}}>\min_{i}\frac{\left(A^{2}\bx\right)_i}{\left(A\bx\right)_i}=\frac{\left(A^2\bx\right)_{i_0}}{\left(A\bx\right)_{i_0}}$ and $a_{i_0,j_0}\neq 0$. Then, $\min_{i}\frac{\left(A^{2}\bx\right)_i}{\left(A\bx\right)_i} > \min_{i}\frac{\left(A\bx\right)_i}{\left(\bx\right)_i}$.
	\end{enumerate}
\end{proposition}

We obtain the following corollaries from Propositions \ref{prop:inequality} and \ref{prop:inequality2}.

\begin{corollary}\label{cor:ineqality2}
	Let $A$ be an $n\times n$ irreducible nonnegative matrix and $\bx\in\mathbb{R}_{++}^n$. Suppose that $\bx$ is not a Perron vector. If there exists an index $i_1\in\{1,\dots,n\}$ such that $\frac{\left(A^2\bx\right)_{i_1}}{\left(A\bx\right)_{i_1}}=\max_{i}\frac{\left(A^{2}\bx\right)_i}{\left(A\bx\right)_i}$ and the $i_1^\text{th}$ row of $A$ consists of nonzero entries, then
	\begin{align*}
		\max_{i}\frac{\left(A^{2}\bx\right)_i}{\left(A\bx\right)_i}< \max_{i}\frac{\left(A\bx\right)_i}{\left(\bx\right)_i}.
	\end{align*}
	Similarly, if there exists an index $i_2\in\{1,\dots,n\}$ such that $\frac{\left(A^2\bx\right)_{i_2}}{\left(A\bx\right)_{i_2}}=\min_{i}\frac{\left(A^{2}\bx\right)_i}{\left(A\bx\right)_i}$ and the $i_2^\text{th}$ row of $A$ consists of nonzero entries, then
	\begin{align*}
		\min_{i}\frac{\left(A\bx\right)_i}{\left(\bx\right)_i}<\min_{i}\frac{\left(A^{2}\bx\right)_i}{\left(A\bx\right)_i}.
	\end{align*}
\end{corollary}

\begin{corollary}\label{cor:ineqality}
	Let $A$ be a positive matrix, and $\bx\in\mathbb{R}_{++}^n$. If $\bx$ is not a Perron vector, then
	\begin{align*}
		\max_{i}\frac{\left(A^{2}\bx\right)_i}{\left(A\bx\right)_i}&< \max_{i}\frac{\left(A\bx\right)_i}{\left(\bx\right)_i},\\
		\min_{i}\frac{\left(A\bx\right)_i}{\left(\bx\right)_i}&<\min_{i}\frac{\left(A^{2}\bx\right)_i}{\left(A\bx\right)_i}.
	\end{align*}
\end{corollary}

Here is a condition for sequences $(a_k(A,\bx))_{k\geq 1}$ and $(b_k(A,\bx))_{k\geq 1}$ to be strictly decreasing and increasing, respectively.

\begin{theorem}\label{thm:UpperMonocity}
	Let $A$ be a positive matrix, and let $\bx$ be a positive vector. Then,
	\begin{enumerate}[label=(\roman*)]
		\item\label{case1} If $A^{r}\bx$ is not a Perron vector for $r\geq 0$, then $(a_k(A,\bx))_{k\geq 1}$ is strictly decreasing, and $(b_k(A,\bx))_{k\geq 1}$ is strictly increasing. Moreover, both converge to $\rho(A)$.
		\item\label{case2} If there exists the minimum integer $r_0\geq 0$ such that $A^{r_0}\bx,A^{r_0+1}\bx,\dots$ are Perron vectors of $A$, then
		\begin{align*}
			&b_1(A,\bx)<\cdots<b_{r_0}(A,\bx)=b_{r_0+1}(A,\bx)=\cdots=\rho(A),\\
			&a_1(A,\bx)>\cdots>a_{r_0}(A,\bx)=a_{r_0+1}(A,\bx)=\cdots=\rho(A).
		\end{align*}
	\end{enumerate}
\end{theorem}
\begin{proof}
	Let $\bx^{(r)}=A^r\bx$ for $r\geq 0$. Suppose that $\bx^{(r)}$ is not a Perron vector for $r\geq 0$. By Corollary \ref{cor:ineqality}, $\max_{i}\frac{\left(A^{2}\bx^{(r-1)}\right)_i}{\left(A\bx^{(r-1)}\right)_i}<\max_{i}\frac{\left(A\bx^{(r-1)}\right)_i}{\left(\bx^{(r-1)}\right)_i}$ and $\min_{i}\frac{\left(A\bx^{(r-1)}\right)_i}{\left(\bx^{(r-1)}\right)_i}<\min_{i}\frac{\left(A^{2}\bx^{(r-1)}\right)_i}{\left(A\bx^{(r-1)}\right)_i}$. Hence, $a_r(A,\bx)>a_{r+1}(A,\bx)$ and $b_r(A,\bx)<b_{r+1}(A,\bx)$. From Remark \ref{Remark:convergence}, the conclusion follows.
\end{proof}

\begin{remark}\label{remark:x is not a perron vector}
	Suppose that $A$ is invertible. If $A^r\by$ is an eigenvector of $A$ for some $r\geq 0$, then $\by$ is an eigenvector of $A$. Hence, if $\bx$ is not an eigenvector, then $A^r\bx$ is not an eigenvector for all $r\geq 0$.
\end{remark}

\begin{example}
	Let $A=\begin{bmatrix}
		2 & 1 & 1\\
		1 & 1 & 1\\
		1 & 1 & 1
	\end{bmatrix}$, and let $\bx=\begin{bmatrix}
		\sqrt{2}-1\\
		\frac{3}{2}-\sqrt{2}\\
		\frac{1}{2}
	\end{bmatrix}$. Then, it can be verified that $\bx$ is not a Perron vector, but $A\bx$ is a Perron vector. By \ref{case2} of Theorem \ref{thm:UpperMonocity}, $$b_1(A,\bx)<b_2(A,\bx)=b_3(A,\bx)=\cdots=\rho(A)=\cdots=a_3(A,\bx)=a_2(A,\bx)<a_1(A,\bx).$$ 
\end{example}

Here is other condition for sequences $(a_k(A,\bx))_{k\geq 1}$ and $(b_k(A,\bx))_{k\geq 1}$ to be strictly monotone and to generate strictly log-concave and log-convex sequences.

\begin{theorem}\label{thm:UpperMonocity2}
	Let $A$ be an $n\times n$ irreducible, nonnegative, positive semidefinite matrix and $\bx\in\RR_{++}^n$. From Proposition \ref{prop:log index}, we may define $i_1$ and $i_2$ to be a log-concavity index and log-convexity index, respectively. We let $g_k=(A^k\bx)_{i_1}$ and $h_k=(A^k\bx)_{i_2}$ for $k\geq 0$. Assume that $\be_{i_1}^TA$ and $\be_{i_2}^TA$ are positive, and $A^r\bx$ is not a Perron vector for $r\geq 0$. Then, the following hold:
	\begin{enumerate}[label=(\roman*)]
		\item There exists $k_1\geq 1$ such that $(a_k(A,\bx))_{k\geq k_1}$ is strictly decreasing and $(g_k)_{k\geq k_1-1}$ is strictly log-concave.
  		\item There exists $k_2\geq 1$ such that $(b_k(A,\bx))_{k\geq k_2}$ is strictly increasing and $(h_k)_{k\geq k_2-1}$ is strictly log-convex.
	\end{enumerate}
\end{theorem}
\begin{proof}
	By Definition \ref{def:log index}, there exist $k_1,k_2\geq 1$ such that for $r\geq k_1$ and $s\geq k_2$, 
	$$\frac{(A^r\bx)_{i_1}}{(A^{r-1}\bx)_{i_1}}=\max_{i}\frac{(A^r\bx)_{i}}{(A^{r-1}\bx)_{i}},\;\text{and}\; \frac{(A^s\bx)_{i_2}}{(A^{s-1}\bx)_{i_2}}=\min_{i}\frac{(A^s\bx)_{i}}{(A^{s-1}\bx)_{i}}.$$
	Let $\bx^{(p)}=A^p\bx$ for $p\geq 0$. Suppose that $\bx^{(p)}$ is not a Perron vector for $p\geq 0$. By Corollary \ref{cor:ineqality2}, $\max_{i}\frac{\left(A^{2}\bx^{(p-1)}\right)_i}{\left(A\bx^{(p-1)}\right)_i}<\max_{i}\frac{\left(A\bx^{(p-1)}\right)_i}{\left(\bx^{(p-1)}\right)_i}$ for $r\geq k_1$ and $\min_{i}\frac{\left(A\bx^{(p-1)}\right)_i}{\left(\bx^{(p-1)}\right)_i}<\min_{i}\frac{\left(A^{2}\bx^{(p-1)}\right)_i}{\left(A\bx^{(p-1)}\right)_i}$ for $s\geq k_2$. 
	By \Cref{Prop:logconcavity index}, the remaining conclusion follows.
\end{proof}


Finally, we shall show the last result of this section, which is a condition for a sequence $(c_k(A,\bx))_{k\geq 1}$ to be strictly increasing and to generate a strictly log-convex sequence.

\begin{lemma}\label{lem:quotientineq}
	Let $A$ be an \( n\times n \) positive definite matrix and \( \bx\in\RR^n\backslash\{\mathbf{0}\} \). Then, for any integer \( r \),
	\begin{align}\label{eq:quotientineq}
		\frac{\bx^TA^r\bx}{\bx^TA^{r-1}\bx}\le \frac{\bx^TA^{r+1}\bx}{\bx^TA^{r}\bx},
	\end{align}
	where the equality holds if and only if $\bx$ is an eigenvector of $A$.
\end{lemma}
\begin{proof}
Since \( A \) is positive definite, there exists an orthonormal basis \( \{\bv_1,\dots,\bv_n\} \) such that \(A\bv_i=\lambda_i\bv_i \) for some $\lambda_i>0$ for $i=1,\dots,n$. Let \( \bx=\sum_{i=1}^n\alpha_i\bv_i \). Then, for any integer \( r \), we have
\begin{align}\label{temp:eqn00}
 \bx^TA^r\bx=\left(\sum_{i=1}^n\alpha_i\bv_i^T\right)\left(\sum_{i=1}^n\alpha_iA^r\bv_i\right)=\left(\sum_{i=1}^n\alpha_i\bv_i^T\right)\left(\sum_{i=1}^n\alpha_i\lambda_i^r\bv_i\right)= \sum_{i=1}^n\alpha_i^2\lambda_i^r.
\end{align}
Note $\alpha_i\ge0$ and $\lambda_i>0$. Using \eqref{temp:eqn00}, one can verify that \eqref{eq:quotientineq} is equivalent to
\[ 
  \left(\sum_{i=1}^n\alpha_i^2\lambda_i^{r-1}\right)\left(\sum_{i=1}^n\alpha_i^2\lambda_i^{r+1}\right)-\left(\sum_{i=1}^n\alpha_i^2\lambda_i^r\right)^2=\sum_{1\le i<j\le n}\alpha_i^2\alpha_j^2\lambda_i^{r-1}\lambda_j^{r-1}\left(\lambda_i-\lambda_j\right)^2\ge0.
\] 

Let us consider the equality. We can find that \( \sum_{1\le i<j\le n}\alpha_i^2\alpha_j^2\lambda_i^{r-1}\lambda_j^{r-1}\left(\lambda_i-\lambda_j\right)^2=0 \) if and only if \( \lambda_s=\lambda_t \) whenever \( \alpha_s\neq0 \) and \( \alpha_t\neq0 \) for any \( s \) and \( t \) with \( s\neq t \). In other words, \( \bx=\sum_{i=1}^n\alpha_i\bv_i \) is a linear combination of eigenvectors corresponding to the same eigenvalue, that is, \( \bx \) is an eigenvector.
\end{proof}

\begin{remark}
	When we relax the condition on $A$ in Lemma \ref{lem:quotientineq} that $A$ is positive semidefinite and $A\neq \mathbf{O}$, one can obtain the same inequality \eqref{eq:quotientineq} by examining the proof, but we fail to obtain the same condition for the equality as in that lemma. For instance, given $A=\begin{bmatrix}
		1 & 0\\
		0 & 0
	\end{bmatrix}$ and $\bx=\begin{bmatrix}
		1 \\ 1
	\end{bmatrix}$, we have $\bx^TA^r\bx=1$ for $r\geq 1$.
\end{remark}

\begin{theorem}\label{Cor:LowerMonocity}
	Let $A$ be a nonnegative, positive definite matrix, and $\bx\in\RR_{+}^n$. If $\bx$ is not a Perron vector, then $(c_k(A,\bx))_{k\ge1}$ is strictly increasing and convergent to $\rho(A)$.
\end{theorem}
\begin{proof}
	The conclusion follows from Lemma \ref{lem:quotientineq} and \Cref{Remark:convergence2}.
\end{proof}

\begin{theorem}\label{Thm:log-concave from c_k}
	Let $A$ be an $n\times n$ nonnegative positive semidefinite matrix and $\bx\in\RR_{+}^n$. Let $s_k=\bx^TA^k\bx$ for $k\geq 0$. Then, $(s_k)_{k\geq 0}$ is log-convex. In particular, if $A$ is positive definite and $\bx$ is not a Perron vector, then $(s_k)_{k\geq 0}$ is strictly log-convex.
\end{theorem}
\begin{proof}
	It follows from \Cref{prop:log-concave from c_k} and \Cref{Cor:LowerMonocity}.
\end{proof}

\section{Combinatorial applications}\label{Sec3:applications}

As seen in Theorems \ref{thm:UpperMonocity}, \ref{thm:UpperMonocity2} and \ref{Cor:LowerMonocity}, upon extra conditions on a nonnegative matrix $A$, we can find (strictly) monotone sequences of lower and upper bounds on the Perron value of $A$ that may induce log-concave or log-convex sequences. With those sequences, we consider two combinatorial applications in this section.

\subsection{Lower and upper bounds on Perron values of rooted trees}\label{Subsec:3.1}

As discussed in \Cref{subsec 1.1}, we shall give lower and upper bounds on Perron values of rooted trees, which can be used for estimating characteristic sets of trees. In \cite{andrade2017combinatorial}, the authors explore such lower and upper bounds with combinatorial interpretation---that is, one can attain those bounds by considering combinatorial objects. Here, we provide sharper bounds in combinatorial settings that one may regard as a generalized version, though it still needs to understand combinatorial interpretation. Besides, we shall verify that bounds obtained from computation of small powers of bottleneck matrices or ``neckbottle matrices'' are sharper than known bounds.

We first consider a weighted connected graph $G$ on $n$ vertices as a general case, and then focus on unweighted rooted trees. Let $M$ be the bottleneck matrix at a vertex $v$ in $G$. It is well known that each principal submatrix of the Laplacian matrix $L(G)$ is non-singular and its inverse is positive definite (see \cite[Chapter 6]{berman1994nonnegative}). Since the Perron value of $M$ is determined by the Perron value of Perron components at $v$, we shall assume that $v$ is not a cut-vertex; otherwise, $M$ would be a block diagonal matrix, which is not irreducible.

\begin{theorem}\label{thm:sequence for M}
	Let $M$ be the bottleneck matrix at a vertex $v$ of a weighted, connected graph $G$ on $n$ vertices. Let $\bx\in\RR_{++}^n$. Suppose that $v$ is not a cut-vertex and $\bx$ is not a Perron vector. Then, $(a_k(M,\bx))_{k\ge1}$ is a strictly decreasing sequence convergent to $\rho(M)$, whereas $(b_k(M,\bx))_{k\ge1}$ and $(c_k(M,\bx))_{k\ge1}$ are strictly increasing sequences convergent to $\rho(M)$.
\end{theorem}
\begin{proof}
	Note that $M$ is a positive matrix. By \Cref{remark:x is not a perron vector}, $M^r\bx$ is not a Perron vector for $r\geq 0$. It is straightforward from Theorems \ref{thm:UpperMonocity} and \ref{Cor:LowerMonocity} with Remarks \ref{Remark:convergence} and \ref{Remark:convergence2}.
\end{proof}

\begin{figure}
	\begin{center}
			\begin{tikzpicture}
				\tikzset{enclosed/.style={draw, circle, inner sep=0pt, minimum size=.10cm, fill=black}}

				\node[enclosed, label={left, yshift=0cm: $1$}] (v_1) at (0,0.9) {};
				\node[enclosed, label={left, xshift=0.1cm: $2$}] (v_2) at (-.6,1.6) {};
				\node[enclosed, label={right, xshift=-.1cm: $4$}] (v_3) at (.6,1.6) {};
				\node[enclosed, label={left, xshift=.1cm: $3$}] (v_4) at (-.6,2.4) {};
				\node[label={left: $\cT$}] (T) at (-1.5,1.7) {};
				\node[label={right: \( N=\begin{bmatrix}
					1 & 1 & 1 & 1 \\
					0 & 1 & 1 & 0 \\
					0 & 0 & 1 & 0 \\
					0 & 0 & 0 & 1 
				\end{bmatrix}\)}] (N) at (1.5,1.7) {};
		
				\draw (v_1) -- (v_2);
				\draw (v_1) -- (v_3);
				\draw (v_2) -- (v_4);
		
				\path[draw=black] (v_1) circle[radius=0.1];
				\end{tikzpicture}
	\end{center}
	\caption{An example for path matrix of $\cT$ with root $1$.}\label{Fig:path matrix}
\end{figure}
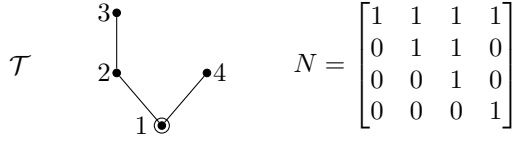


Let $M$ be the bottleneck matrix of a rooted tree $\cT$ with vertex set $\{1,\dots,n\}$ and root $x$. We shall introduce other object having the same Perron value as $\rho(M)$ that we can also use for approximating $\rho(M)$. The \textit{path matrix} of $\cT$, denoted by \( N \), is the matrix whose $j^\text{th}$ column is the $(0,1)$ vector where the $i^\text{th}$ component is $1$, if $i$ lies on the path from $j$ to $x$, and $0$ otherwise. See Figure \ref{Fig:path matrix} for an example. Then, $M$ can be written as $M=N^TN$ (see \cite{andrade2017combinatorial}). Appropriately labelling the vertices, we can find that the matrix $N$ is upper triangular with ones on the main diagonal and the row of $N$ corresponding to $x$ is the all ones vector. The matrix $Q=NN^T$ is called the \textit{neckbottle matrix} of $\cT$, which was introduced in \cite{ciardo2021perron}. Then, $Q$ is positive definite, but is not a positive matrix.

\begin{theorem}\label{thm:sequence for Q}
	Let $\cT$ be a rooted tree on $n$ vertices and $Q$ be its neckbottle matrix. Let $\bx\in\RR_{++}^n$. Suppose that $\bx$ is not a Perron vector. Then, \( (c_k(Q,\bx))_{k\ge1} \) is a strictly increasing sequence convergent to $\rho(Q)$.
\end{theorem}
\begin{proof}
	 It follows from Theorem \ref{Cor:LowerMonocity}.
\end{proof}


\begin{remark}
	Let $M$ be the bottleneck matrix of a rooted tree. In \cite{andrade2017combinatorial}, $\|M\|_1$ appears as an upper bound on $\rho(M)$, where $\|\cdot\|_1$ is the $\ell_1$ matrix norm. Note that $\|M\|_1=\max_i\frac{(M\mathbf{1})_i}{(\mathbf{1})_i}=a_1(M,\mathbf{1})$. We can see from \Cref{thm:sequence for M} that $a_k(M,\mathbf{1})<\|M\|_1$ for $k\geq 2$. 
\end{remark}

\begin{remark}
	Let $N$ and $M$ be the path matrix and the bottleneck matrix of a rooted tree, respectively. In \cite{andrade2017combinatorial}, two lower bounds on $\rho(M)$ appear: one is the \textit{combinatorial Perron value} $\rho_c(N)$ given by $\rho_c(N)=\frac{\mathbf{1}^T(NN^T)^2\mathbf{1}}{\mathbf{1}^TNN^T\mathbf{1}}$, and the other is given by $\pi(N)=\left(\frac{\mathbf{1}^T(NN^T)^3\mathbf{1}}{\mathbf{1}^TNN^T\mathbf{1}}\right)^\frac{1}{2}$. Let $Q$ be the neckbottle matrix. Then, $\rho_c(N)=\frac{\mathbf{1}^TQ^2\mathbf{1}}{\mathbf{1}^TQ\mathbf{1}}=c_2(Q,\mathbf{1})$ and $\pi(N)=\left(\frac{\mathbf{1}^TQ^3\mathbf{1}}{\mathbf{1}^TQ\mathbf{1}}\right)^{1/2}$. By \Cref{thm:sequence for Q}, we have $\frac{\mathbf{1}^TQ^2\mathbf{1}}{\mathbf{1}^TQ\mathbf{1}}<\frac{\mathbf{1}^TQ^3\mathbf{1}}{\mathbf{1}^TQ^2\mathbf{1}}$. Multiplying both sides by $\frac{\mathbf{1}^TQ^2\mathbf{1}}{\mathbf{1}^TQ\mathbf{1}}$ and taking the square root of both sides yield $\rho_c(N)<\pi(N)$. Hence, 
	$$\frac{\mathbf{1}^TQ^2\mathbf{1}}{\mathbf{1}^TQ\mathbf{1}}< \left(\frac{\mathbf{1}^TQ^3\mathbf{1}}{\mathbf{1}^TQ\mathbf{1}}\right)^{1/2}=\left(\frac{\mathbf{1}^TQ^2\mathbf{1}}{\mathbf{1}^TQ\mathbf{1}}\frac{\mathbf{1}^TQ^3\mathbf{1}}{\mathbf{1}^TQ^2\mathbf{1}}\right)^{1/2}< \frac{\mathbf{1}^TQ^3\mathbf{1}}{\mathbf{1}^TQ^2\mathbf{1}}.$$
	Therefore, $\rho_c(N)<\pi(N)<c_k(Q,\mathbf{1})$ for $k\geq 3$.
\end{remark}

In addition to the bounds in the above remarks, the author of \cite{molitierno2018tight} investigated a tight upper bound on Perron values of ``rooted brooms'' with some particular root, by virtue of the fact that for any unweighted and connected graph, Perron values of bottleneck matrices at vertices with the same eccentricity as that of the root in the broom are bounded below. Here we provide a tighter upper bound than the one in \cite{molitierno2018tight} except for a few cases. Moreover, we give lower bounds on Perron values of rooted trees, provided their roots have the same eccentricity. 

Let $B(d,r)$ denote the tree, called a \textit{broom}, formed from a path on $d$ vertices by adding $r$ pendent vertices to an end-vertex $v$ of the path. If $v$ is designated as the root, then we use $B_1(d,r)$ to denote the corresponding rooted tree (see \Cref{Figure:Broom2}); and if the other end-vertex of the path chosen as the root, then $B_2(d,r)$ denotes the resulting rooted tree (see \Cref{Figure:Broom1}). 

For two nonnegative matrices $A$ and $B$, we use the notation $A\geq B$ to mean that $A$ is entry-wise greater than or equal to $B$, that is, $A-B$ is nonnegative. If $A\geq B$ then $\rho(A)\geq\rho(B)$ (see \cite{horn2012matrix}). Let $M$ be the bottleneck matrix of a rooted tree on $n$ vertices such that the eccentricity of its root is $d$. Let $r=n-d$. Let $M_1$ and $M_2$ be the bottleneck matrices of $B_1(d,r)$ and $B_2(d,r)$, respectively. As mentioned in \cite{kirkland1997algebraic}, one can prove $M\leq M_2$ by using induction on $r$; similarly, it can be also verified that $M_1\leq M$. 

We now provide a sharper upper bound on $\rho(M_2)$, and lower bounds on $\rho(M_1)$.



\begin{theorem}\label{thm:upperForPerron}
	Let $G$ be an unweighted, connected graph on $n+1$ vertices with a vertex $v$, and let $M$ be the bottleneck matrix at $v$. Suppose that $d+1$ is the eccentricity of $v$. Let $M_2$ be the bottleneck matrix of $B_2(d,r)$ where $r=n-d$. Then, an upper bound $a_3(M_2,\mathbf{1})$ on $\rho(M)$ is given by
	\begin{align}\label{UpperBound for rho T}
		a_3(M_2,\mathbf{1})=\frac{S_1(d,r)}{S_2(d,r)},
	\end{align}
	where
	\begin{align*}
		S_1(d,r)=&d^3r^3 +\frac{1}{6}d^2(7d^2 + 9d + 20)r^2+\frac{1}{120}d(61d^4 + 140d^3 + 315d^2 + 280d + 404)r\\
		&+\frac{1}{720}(61d^6 + 183d^5 + 385d^4 + 465d^3 + 634d^2 + 432d + 720),\\
		S_2(d,r)=&d^2r^2 +\frac{1}{6}d(5d^2 + 6d + 13)r+\frac{1}{24}(5d^4 + 10d^3 + 19d^2 + 14d + 24).
	\end{align*}
\end{theorem}
\begin{proof}
	It is known in \cite{molitierno2018tight} that $\rho(M)\leq \rho(M_2)$. It is straightforward from Theorem \ref{thm:sequence for M} that $\rho(M_2)< a_3(M_2,\mathbf{1})$. See \Cref{subappen:1} for the completion of the proof.
\end{proof}

\begin{remark}\label{Remark:sharper}
	Continuing \Cref{thm:upperForPerron}, if $d\geq 5$ and $r>\frac{2d^5+37d^4+5d^3-395d^2-376d-2}{16d^4-260d^2-360d-116}\approx\frac{d}{8}$, then the following upper bound for $\rho(M)$ appears in \cite{molitierno2018tight}:
	\begin{align*}
		f(d,r)=dr+\frac{4d^4+20d^3+25d^2+40d+1}{10d^2+45d+5}\geq \rho(M).
	\end{align*}
	We shall show that $a_3(M,\mathbf{1})$ is a shaper bound than $f(d,r)$ except for a few values of $d$ and $r$ with $r\geq \frac{d}{8}$. Using MATLAB\textsuperscript{\textregistered}, we can find that
	\begin{align*}
		f(d,r)-a_3(M,\mathbf{1})=\frac{N(d,r)}{D(d,r)},
	\end{align*}
	where
	\begin{align*}
		N(d,r)=&24d^2(2d + 1)(d + 1)(2d^2 - 3d - 29)r^2\\
		&+12d(2d + 1)(d + 1)(d^2 - d + 4)(2d^2 - 3d - 29)r\\
		&-(d + 3)(d + 1)(2d^6 + 67d^5 - 202d^4 - 131d^3 - 568d^2 + 100d + 192),\\		
		D(d,r)=&30(2d^2 + 9d + 1)(5d^4 + 20d^3r + 10d^3 + 24d^2r^2 + 24d^2r + 19d^2 + 52dr + 14d + 24).
	\end{align*}
	Clearly, $D(d,r)>0$ and $N(d,r)$ is a quadratic polynomial in $r$. One can verify that the coefficients of $r$ and $r^2$ in $N$ are positive, and the constant is negative for $d\geq 5$. This implies that if \( d\ge5 \) and $N(d,r_0)>0$ for some $r_0$, then $N(d,r)>N(d,r_0)$ for $r\geq r_0$. In terms of $d$, one can verify that $N(d,d/8)>0$ for $d\geq 17$. Hence, $a_3(M,\mathbf{1})$ is a sharper upper bound on \( \rho(M) \) for $d\geq 17$. Furthermore, one can check that given $5\leq d \leq 16$, there exists $2\le r_0\le 4$ such that $a_3(M,\mathbf{1})<f(d,r)$ for $r\geq r_0$.
\end{remark}


\begin{theorem}\label{thm:lowerForPerron}
	Let $\cT$ be a rooted tree on $n$ vertices, and $M$ be the bottleneck matrix of $\cT$. Suppose that $d$ is the eccentricity of the root. Let $r=n-d$. Suppose that $Q_1$ and $M_1$ are the neckbottle matrix and the bottleneck matrix of $B_1(d,r)$, respectively. (Note $\rho(M_1)=\rho(Q_1)$.) Then, two lower bounds $c_3(Q_1,\mathbf{1})$ and $c_3(M_1,\mathbf{1})$ on $\rho(M)$ are given by $c_3(Q_1,\mathbf{1})=\frac{U_1(d,r)}{U_2(d,r)}$ and $c_3(M_1,\mathbf{1})=\frac{V_1(d,r)}{V_2(d,r)}$ where
	\begin{align*}
		U_1(d,r)=&4r^3 + 2(d^2 + 3d + 4)r^2 +\frac{1}{12}(13d^4 + 26d^3 + 41d^2 + 28d + 48)r\\
		&+\frac{1}{2520}d(d+1)(2d+1)(68d^4 + 136d^3 + 133d^2 + 65d + 18),\\
		U_2(d,3)=&4r^2 + 2(d^2 + d + 2)r +\frac{1}{30}d(d+1)(2d+1)(2d^2 + 2d + 1),\\
		V_1(d,r)=&r^4 + (4d + 3)r^3 +\frac{1}{2}(2d+3)(d+2)(d+1)r^2\\
		&+\frac{1}{15}(d+1)(4d^4 + 16d^3 + 19d^2 + 21d + 15)r\\
		&+\frac{1}{2520}d(2d+1)(d+1)(68d^4 + 136d^3 + 133d^2 + 65d + 18),\\
		V_2(d,r)=&r^3 + (3d + 2)r^2 +\frac{1}{3}(d+1)(2d^2+4d+3)r\\
		&+\frac{1}{30}d(2d+1)(d+1)(2d^2+2d+1).
	\end{align*}
	Moreover, for each $d\geq 3$, the polynomial \( c_3(M_1,\mathbf{1})-c_3(Q_1,\mathbf{1}) \) in \( r \) has a root $r_0$ in the open interval \( (0.4d^2-1,0.42d^2+2) \); further, $c_3(M_1,\mathbf{1})<c_3(Q_1,\mathbf{1})$ for $r<r_0$ and $c_3(M_1,\mathbf{1})>c_3(Q_1,\mathbf{1})$ for $r>r_0$.
\end{theorem}
\begin{proof}
	By Theorems \ref{thm:sequence for M} and \ref{thm:sequence for Q}, we obtain $c_3(Q_1,\mathbf{1})< \rho(M)$ and $c_3(M_1,\mathbf{1})<\rho(M)$. See \Cref{subappen:2} for the completion of the proof.
\end{proof}

\begin{remark}\label{remark:Comparison for lowerbounds}
	As another lower bound on the Perron value of a rooted tree with root $x$, one may consider $b_k(M,\mathbf{1})=\min\limits_{i}{\frac{(M^k\mathbf{1})_i}{(M^{k-1}\mathbf{1})_i}}$ for $k\geq 1$, where $M$ is the bottleneck matrix of the rooted tree. Since the row and column of $M$ corresponding to $x$ are the all ones vector, we can see from Theorem \ref{thm:sequence for M} that
	\begin{align*}
		b_k(M,\mathbf{1})=\min\limits_{i}{\frac{(M^k\mathbf{1})_i}{(M^{k-1}\mathbf{1})_i}}\leq \frac{(M^k\mathbf{1})_x}{(M^{k-1}\mathbf{1})_x}=\frac{\be_x^T M^k\mathbf{1}}{\be_x^T M^{k-1}\mathbf{1}}=\frac{\mathbf{1}^T M^{k-1}\mathbf{1}}{\mathbf{1}^T M^{k-2}\mathbf{1}}<\frac{\mathbf{1}^T M^k\mathbf{1}}{\mathbf{1}^T M^{k-1}\mathbf{1}}=c_k(M,\mathbf{1}).
	\end{align*}
	Hence, $c_k(M,\mathbf{1})$ is a sharper lower bound on $\rho(M)$ than $b_k(M,\mathbf{1})$.
\end{remark}

\begin{example}\label{Observation:comparison}
	Let $M$ be the bottleneck matrix of $B_1(16,r)$. As seen in Figure \ref{fig:comparison}, the sharpness of $c_3(Q,\mathbf{1})$ and $c_3(M,\mathbf{1})$, as lower bounds, is inverted at a particular value of $r$. By Theorem \ref{thm:lowerForPerron}, $c_3(M,\mathbf{1})-c_3(Q,\mathbf{1})$ has exactly one positive root $r_0$ in the open interval $(101.4,109.52)$. Indeed, using MATLAB\textsuperscript\textregistered, we have $r_0\approx 108.1708$.
	\begin{figure}[h!]
		\begin{center}
			\includegraphics[width=0.8\textwidth]{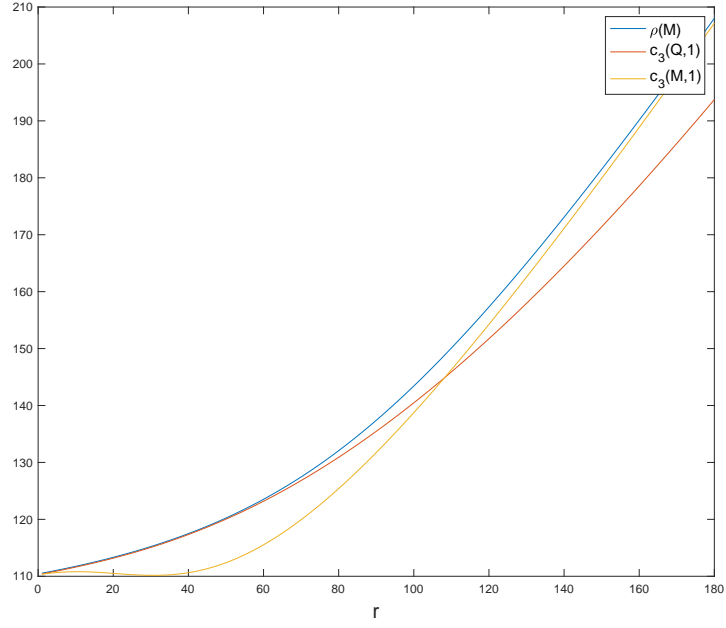}
		\end{center}
		\caption{Comparison of lower bounds $c_3(M,\mathbf{1})$ and $c_3(Q,\mathbf{1})$ for $\rho(M)$.}\label{fig:comparison}
	\end{figure}
\end{example}

\subsection{Log-convexity and log-concavity of recurrence relations}\label{Subsec:3.2}

As seen in \Cref{thm:logconcave seq} and \Cref{Thm:log-concave from c_k}, we can construct
log-convex and log-concave sequences. In this subsection, we examine various examples. 

\begin{example}\label{ex:path bottleneck}
	Let $n\geq 2$. Let $P_n$ be the rooted path on vertices $1,\dots,n$, where $i$ is adjacent to $i+1$ for $i=1,\dots,n-1$ and vertex $1$ is the root. We use \( M_{P_n} \) to denote the bottleneck matrix of the rooted path. Then,
	\[
		M_{P_n}=\begin{bmatrix}
		 1 &  &  & \cdots &  &  1 \\
		  & 2 &  & \cdots &  &  2 \\
		  &  & 3 & \cdots &  &  3 \\
		  \vdots & \vdots & \vdots & \ddots &  & \vdots  \\
		  &  &  &  & n-1 & n-1  \\
		1  & 2 & 3 &  & n-1 & n \\
	 \end{bmatrix}.
	\]
For the path matrix $N$, we have
\begin{align*}
	M_{P_n}=N^TN=\begin{bmatrix}
		1 &  &  \mathbf{O}   \\
		\vdots &  \ddots & \\
		1 &\cdots  &  1  \\
   \end{bmatrix}\begin{bmatrix}
		 1 & \cdots &  1 \\
		  & \ddots &  \vdots \\
		  \mathbf{O} &  &  1 \\
	\end{bmatrix}.
\end{align*}
Let \( s_k^{(n,i,j)}=\be_i^TM_{P_n}^k\be_j \) for some \( 1\le i,j\le n \). Regarding $N$ as the adjacency matrix of a directed graph, interpreting $\be_i^T(N^TN)^k\be_j$ as the number of walks of length $2k$ from $i$ to $j$ on the directed graph with reversing all the directions after each step of walk, one can find that \( s_k^{(n,i,j)} \) is the number of sequences \( (a_0,\dots,a_{2k-2}) \), where \( 1\le a_t\le n \) for \( t=0,\dots,2k-2 \), satisfying \( i\ge a_0\le a_1\ge\dots\le a_{2k-3}\ge a_{2k-2}\le j \). Note by symmetry that \( s_k^{(n,i,j)}=s_k^{(n,j,i)} \) for any \( i,j,k,n \). By \Cref{Thm:log-concave from c_k}, the combinatorial sequence \( \left(s_k^{(n,j,j)}\right)_{k\ge0} \) is log-convex for any \( 1\le j\le n \). (In particular, \( s_k^{(2,1,1)}=F_{2k-1} \), where \( F_{k} \) is the \( k^\text{th} \) Fibonacci number with \( F_0=0,F_1=1 \).) 
\end{example} 

Given an $n\times n$ irreducible nonnegative matrix $A$ with $\bx\in\RR_{++}^n$, the sequences $(a_k(A,\bx))_{k\geq 1}$ and $(b_k(A,\bx))_{k\geq 1}$ may produce log-concave and log-convex sequences, respectively.


\begin{example}
	Let \( A=\begin{bmatrix}
		a &  b \\
		c &  d \\
	\end{bmatrix}
	 \) be irreducible and nonnegative, and \( \bx=\begin{bmatrix}
		x_1   \\
		x_2   \\
	 \end{bmatrix}
	  \) be positive. Let $k\geq 1$. Using induction on $k$, we can find that 
	  \begin{align*}
		  D(k)=\left(A^k\bx\right)_{2}\left(A^{k-1}\bx\right)_{1}-\left(A^k\bx\right)_{1}\left(A^{k-1}\bx\right)_{2}=(x_1x_2(d-a)-bx_2^2+cx_1^2)(ad-bc)^{k-1}.
	  \end{align*}
	  Then, $\det (A)\geq 0$ if and only if $D(k)$ is either nonpositive or nonnegative for all $k\geq 1$. Let $g_k=(A^k\bx)_1$ and $h_k=(A^k\bx)_2$. It follows from \Cref{Prop:logconcavity index} that $\det (A)\geq 0$ if and only if one of two sequences $(g_k)_{k\geq 1}$ and $(h_k)_{k\geq 1}$ is log-concave and the other is log-convex, upon the sign of $x_1x_2(d-a)+bx_2^2-cx_1^2$. This shows the existence of log-concavity and convexity indices for \( 2\times2 \) irreducible nonnegative matrix \( A \) even if \( A \) is not symmetric. (We note that \Cref{prop:log index} only provides the existence of log-concavity and convexity indices when $A$ is positive semidefinite.) As a concrete example, consider $A=\begin{bmatrix}
		2 &  1 \\
		2 &  2 
	\end{bmatrix}$ and $\bx=\begin{bmatrix}
		1 \\ 1
	\end{bmatrix}$. Then, $\det(A)>0$ and \( D(k)>0 \) for all \( k \). Moreover, $(g_k)_{k\geq 1}$ is log-convex and $(h_k)_{k\geq 1}$ is log-concave.
\end{example}

\begin{remark}
	Given a \( k\times k \) matrix \( A \),
	 let \( p(x) \) be the characteristic polynomial of \( A \). Since \( p(A)=\mathbf{O} \), multiplying \( \bx \) on the right of both sides, we obtain a recurrence relation for \( r_n^{(i)}=(A^n\bx)_{i} \). If \( A \) is a \( 2\times 2 \) matrix where \( {\rm tr}(A)>0 \) and \( \det(A)>0 \), then \( r_n^{(i)} \) satisfies a recurrence relation \( r_{n+1}^{(i)}=c_1r_n^{(i)}+c_2r_{n-1}^{(i)} \) for some positive \( c_1 \) and \( c_2 \). Then we can determine whether \( r_n^{(i)} \) is log-convex or log-concave by \cite{liu2007log}. Indeed, if $\left( r_0^{(i)},r_1^{(i)},r_2^{(i)} \right)$ is log-convex (resp. log-concave), then $\left(r_n^{(i)} \right)_{n\ge0}$ is log-convex (resp. log-concave) for each \( i \).
\end{remark}

\begin{example}
	Let \( A=\begin{bmatrix}
		2 & 1 &  1 \\
		1 & 2 &  0 \\
		1 & 0 & 2 \\
	\end{bmatrix} \) and \( \bx=\begin{bmatrix}
		1\\1\\1
	\end{bmatrix} \). Then, \( A \) is irreducible, nonnegative, positive semidefinite matrix. Its orthonormal eigenvectors are given by
	\[ 
	 \bu_1=\begin{bmatrix}
		1/\sqrt{2}\\1/2\\1/2
	 \end{bmatrix},
	 \bu_2=\begin{bmatrix}
		0\\-1/\sqrt{2}\\1/\sqrt{2}
	 \end{bmatrix},
	 \bu_3=\begin{bmatrix}
		-1/\sqrt{2}\\1/2\\1/2
	 \end{bmatrix}
	\]
	and hence 
	\[ 
		\by_1=\bu_1\bu_1^T\bx=\begin{bmatrix}
			\frac{1+\sqrt{2}}{2}\\
			\frac{2+\sqrt{2}}{4}\\
			\frac{2+\sqrt{2}}{4}
		\end{bmatrix},
		\by_2=\bu_2\bu_2^T\bx=\begin{bmatrix}
			0\\0\\0
		\end{bmatrix},
		\by_3=\bu_3\bu_3^T\bx\begin{bmatrix}
			\frac{1-\sqrt{2}}{2}\\
			\frac{2-\sqrt{2}}{4}\\
			\frac{2-\sqrt{2}}{4}
		\end{bmatrix}.
	\]
	Following the notation in the proof of Proposition~\ref{prop:log index}, we have \( j_0=3 \) since \( \by_2=\mathbf{0}\). One can directly have that \( p_0=1 \) and \( X_1=\{p_0\} \). Furthermore, 
	\[ 
		F(p,q,k)=(\mu_1\mu_3)^{k-1}(\mu_1-\mu_3)\det\begin{bmatrix}
		y_{p,1} & y_{p,3}\\
		y_{q,1} & y_{q,3}
	\end{bmatrix}, 
	\]
	 where \( \mu_1=2+\sqrt{2} \) and \( \mu_3=2-\sqrt{2} \). Since \( F(1,q,k)>0 \) for \( q=2,3 \) and any \( k \), the log-concavity index is \( 1 \). The produced log concave sequence is 
	\[ 
	1, 4, 14, 48, 164, 560, 1912,\cdots. 
	\]
	For more details, See A007070 in OEIS~\cite{OEIS}. In the same way, one can find that \( F(2,1,k)<0, F(2,3,k)=0 \) for any \( k \) and hence the log-convexity index is \( 2 \) and \( 3 \). (They produce the same sequence.) The produced log convex sequence is 
	\[ 
	1, 3, 10, 34, 116, 396, 1352, \cdots,
	\]
	see A007052 in OEIS.
\end{example}
\begin{example}
	In this example, we consider a family of particular symmetric tridiagonal matrices. Let $A=aI+bP$ for some $a,b>0$ where $P$ is the adjacency matrix of the path graph on $n$ vertices.  Let a Motzkin path be a lattice path using the step set \( \{up=(1,1),level=(1,0),down=(1,-1)\} \) that never goes below the \( x \)-axis. Then, the $i^\text{th}$ entry of $A^k\mathbf{1}$ is the sum of weights of weighted Motzkin paths of length \( k \) starting at \( (0,i) \)
	, where up and down steps are of weight \( a \) and the level step has the weight \( b \). 
	
	It can be found in \cite{brouwer2011spectra} that eigenvalue $\lambda_l$ of $P$ for $l=1,\dots,n$ is given by $\lambda_l=2\cos\left(\frac{l\pi}{n+1}\right)$, and the corresponding eigenvector $\bu_l$ is given by $(\bu_l)_j=\sin\left(\frac{lj\pi}{n+1}\right)$. So, the eigenvalues of $A$ are given by $a+b\lambda_l$ for $l=1,\dots,n$. Hence, if $a\geq 2b$, then $A$ is irreducible and positive semidefinite. Therefore, by \Cref{thm:logconcave seq}, $(a_k(A,\mathbf{1}))_{k\geq 1}$ and $(b_k(A,\mathbf{1}))_{k\geq 1}$ produce log-concave and log-convex sequences.

	We now determine log-concavity and log-convexity indices, considering the case \ref{case b} in \Cref{Remark:finding index}. All eigenvalues of $A$ are distinct and $\lambda_1>\dots>\lambda_n$. It can be checked that $\bu_{l}^T\mathbf{1}\geq 0$ with the equality if $l$ is even. Let $\by_l = (\bu_{1}^T\mathbf{1})\bu_l$ for $l=1,\dots,n$. We can find from $\sin(3\theta) = -4\sin^3(\theta)+3\sin(\theta)$ that for $j=2,\dots,n-1$,
	\begin{align*}
		\det\begin{bmatrix}
			(\by_1)_1 & (\by_3)_1 \\
			(\by_1)_j & (\by_3)_j \\
		\end{bmatrix} = (\bu_{1}^T\mathbf{1})(\bu_{3}^T\mathbf{1})\det\begin{bmatrix}
			\sin\left(\frac{\pi}{n+1}\right) & \sin\left(\frac{3\pi}{n+1}\right) \\
			\sin\left(\frac{j\pi}{n+1}\right) & \sin\left(\frac{3j\pi}{n+1}\right)
		\end{bmatrix}<0.
	\end{align*}
	Furthermore, $\det\begin{bmatrix}
		(\by_1)_1 & (\by_j)_1 \\
		(\by_1)_n & (\by_j)_n \\
	\end{bmatrix}=0$ for $j=2,\dots,n$. Therefore, $1$ and $n$ are the log-convexity indices of $A$ associated with $\mathbf{1}$. Similarly, one can verify that $\lfloor\frac{n}{2}\rfloor$ and $\lceil\frac{n}{2}\rceil$ are the log-concavity indices of $A$ associated with $\mathbf{1}$.
\end{example}

\begin{remark}
	To find the log-convexity and concavity indices, we have to calculate eigenvalues and eigenvectors. It seems necessary to develop a tool to obtain log convexity and concavity indices with simple hand calculations.
   \end{remark}

\section{Appendices}\label{appendix}

\appendix

\section{Proofs for some results in Subsection \ref{Subsec:3.1}}
\label{appendix:1}

We remark that tedious calculations based on recurrence relations in this appendix will not be displayed in detail, and they are performed by MATLAB\textsuperscript{\textregistered}. We denote by \( M_{P_d} \) the bottleneck matrix of the rooted path in Example~\ref{ex:path bottleneck}. 


\subsection{Proof pertaining to the upper bound in Theorem \ref{thm:upperForPerron}}\label{subappen:1}

Here, we complete the proof of the remaining argument in Theorem \ref{thm:upperForPerron}.

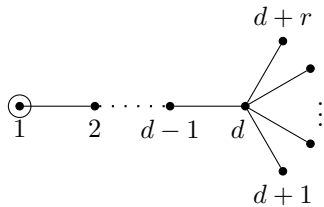
\begin{figure}[h!]
	\begin{center}
		\begin{tikzpicture}
		\tikzset{enclosed/.style={draw, circle, inner sep=0pt, minimum size=.10cm, fill=black}}
		\node[enclosed, label={below, yshift=0cm: $1$}] (v_1) at (0,1) {};
		\node[enclosed, label={below, yshift=0cm: $2$}] (v_2) at (1,1) {};
		\node[enclosed, label={below, yshift=0cm: $d-1$}] (v_3) at (2,1) {};
		\node[enclosed, label={below, xshift=-.1cm: $d$}] (v_4) at (3,1) {};
		\node[enclosed,  label={below, yshift=0cm: $d+1$}] (v_6) at (3.5,0.134) {};
		\node[enclosed] (v_7) at (3.866,0.5) {};
		\node[enclosed,label={above, yshift=0cm: $d+r$}] (v_8) at (3.5,1.8660) {};
		\node[enclosed] (v_9) at (3.8660,1.5) {};
		\node[label={center, yshift=0cm: $\vdots$}] (v_10) at (4,1) {};
		
		\draw (v_1) -- (v_2);
		\draw[thick, loosely dotted] (v_2) -- (v_3);
		\draw (v_3) -- (v_4);

		\draw (v_4) -- (v_6);
		\draw (v_4) -- (v_7);
		\draw (v_4) -- (v_8);
		\draw (v_4) -- (v_9);

		\path[draw=black] (v_1) circle[radius=0.15];
		\end{tikzpicture}
	\end{center}
	\caption{A visualization of $B_2(d,r)$ with root \( 1 \).}\label{Figure:Broom1}
\end{figure}

Let $d$ and $r$ be positive integers, and let $n=d+r$. Let $M$ be the bottleneck matrix of $B_2(d,r)$ on $n$ vertices in Figure \ref{Figure:Broom1}. Then $M$ is given by
\begin{align*}
	M=\begin{bmatrix}
		M_{P_d} & M_{21}^T\\
		M_{21} & M_{22}
	\end{bmatrix},
\end{align*}
where $M_{21}=\mathbf{1}_r\begin{bmatrix}
	1 & 2 & \cdots & d
\end{bmatrix}$, and $M_{22}=dJ_r+I_r$. Let $m_i^{(0)}=1$ for $i=1,\dots,n$. Suppose that for $k\geq 1$, 
$$M\begin{bmatrix}
	m_1^{(k-1)}\\ \vdots \\ m_n^{(k-1)}
\end{bmatrix}=\begin{bmatrix}
	m_1^{(k)}\\ \vdots \\ m_n^{(k)}
\end{bmatrix}.$$ 
From the structure of $M$, it can be readily checked that $m_{d+1}^{(k)}=\cdots=m_{n}^{(k)}$ for $k\geq 0$. Furthermore, we can see that
\begin{align}\label{recurrence1}
	m_{l}^{(k)}=\begin{cases*}
		lrm_{d+1}^{(k-1)}+\sum\limits_{i=1}^l\sum\limits_{j=i}^d m_{j}^{(k-1)}& \text{if $1\leq l\leq d$,}\\
		m_{d}^{(k)}+m_{d+1}^{(k-1)}& \text{if $d+1\leq l\leq n$.} 
	\end{cases*}
\end{align}


We claim that \begin{align*}
	\frac{\left(M^3\mathbf{1}\right)_{n}}{\left(M^{2}\mathbf{1}\right)_{n}}=\max_{i}{\frac{\left(M^3\mathbf{1}\right)_{i}}{\left(M^{2}\mathbf{1}\right)_{i}}}.
\end{align*}
To verify \( \left(M^3\mathbf{1}\right)_{l+1}\left(M^{2}\mathbf{1}\right)_{l}-\left(M^3\mathbf{1}\right)_{l}\left(M^{2}\mathbf{1}\right)_{l+1}>0 \), we first consider the case \( 1\le l\le d-1 \). Set \( d=l+k \) for $k\ge1$. With the aid of MATLAB\textsuperscript\textregistered, we can find from the recurrence relations \eqref{recurrence1} that 
\begin{align*}
	&\left(M^3\mathbf{1}\right)_{l+1}\left(M^{2}\mathbf{1}\right)_{l}-\left(M^3\mathbf{1}\right)_{l}\left(M^{2}\mathbf{1}\right)_{l+1}\\
	=&m_{l+1}^{(3)}m_{l}^{(2)}-m_{l}^{(3)}m_{l+1}^{(2)}\\
	=&\frac{1}{360}l(l+1)(l+k)(7l^3+(35k-7)l^2+(40k^2-5k+2)l+20k^2-2)r^3\\
	&+\frac{1}{720}l(l+1)(6l^5+(56k+8)l^4+(168k^2+70k+76)l^3+(200k^3+177k^2+305k-26)l^2\\
	&+(80k^4+160k^3+313k^2+55k-40)l+40k^4+60k^3+122k^2-6k-24)r^2\\
	&+\frac{1}{2880}l(l+1)(3l^6+(48k+9)l^5+(224k^2+120k+61)l^4+(448k^3+448k^2+440k+107)l^3\\
	&+(400k^4+672k^3+972k^2+540k+560)l^2+(128k^{5}+400k^{4}+768k^{3}+748k^{2}+1084k\\
	&+508)l+64k^{5}+160k^{4}+272k^{3}+248k^{2}+456k+192)r\\
	&+\frac{1}{8640}kl\left( 2k+l\right)
	\left( l+1\right) \left( 2k+l+1\right)(9l^{4}+\left( 36k+18\right) l^{3}+\left( 44k^{2}+54k+25\right)l^{2}\\
	&+\left( 16k^{3}+44k^{2}+50k+16\right) l+8k^{3}+20k^{2}+16k+4).
\end{align*}
We can view $m_{l+1}^{(3)}m_{l}^{(2)}-m_{l}^{(3)}m_{l+1}^{(2)}$ as a polynomial in $r$. Let $C(i)$ be the coefficient of $r^i$ for $0\leq i\leq 3$. One can easily check that for \( k\ge1\) and \( l\ge1 \), $C(i)>0$ for each $i=0,\dots,3$. Thus, for $1\leq l\leq d-1$, we have  
$$
\left(M^3\mathbf{1}\right)_{l+1}\left(M^{2}\mathbf{1}\right)_{l}-\left(M^3\mathbf{1}\right)_{l}\left(M^{2}\mathbf{1}\right)_{l+1}>0.
$$
Moreover, if \( l=d \), then
\begin{align*}
	&\left(M^3\mathbf{1}\right)_{d+1}\left(M^{2}\mathbf{1}\right)_{d}-\left(M^3\mathbf{1}\right)_{d}\left(M^{2}\mathbf{1}\right)_{d+1}\\
	=&\frac{1}{360}d^2(d-1)(d+1)(7d^2+2)r^2\\
	&+\frac{1}{360}d(d-1)(d+1)(3d^4 + 7d^3 + 45d^2 + 32d + 12)r\\
	&+\frac{1}{2880}d(d+1)(3d^6 + 9d^5 + 61d^4 + 107d^3 + 560d^2 + 508d + 192)>0.
\end{align*}
Recall $m_{d+1}^{(k)}=\cdots=m_{n}^{(k)}=\left(M^k\mathbf{1}\right)_{n}$ for $k\geq 0$. Therefore,
\begin{align*}
	\frac{\left(M^3\mathbf{1}\right)_{n}}{\left(M^{2}\mathbf{1}\right)_{n}}=\max_{i}{\frac{\left(M^3\mathbf{1}\right)_{i}}{\left(M^{2}\mathbf{1}\right)_{i}}},
\end{align*}
where
\begin{align*}
	(M^3\mathbf{1})_n
	=&d^3r^3 +\frac{1}{6}d^2(7d^2 + 9d + 20)r^2+\frac{1}{120}d(61d^4 + 140d^3 + 315d^2 + 280d + 404)r\\
	&+\frac{1}{720}(61d^6 + 183d^5 + 385d^4 + 465d^3 + 634d^2 + 432d + 720),
\end{align*}
and 
\begin{align*}
	(M^2\mathbf{1})_n=&d^2r^2 +\frac{1}{6}d(5d^2 + 6d + 13)r+\frac{1}{24}(5d^4 + 10d^3 + 19d^2 + 14d + 24).
\end{align*}

\subsection{Proofs pertaining to the lower bounds in Theorem \ref{thm:lowerForPerron}}\label{subappen:2}
We now complete the remaining argument in Theorem \ref{thm:lowerForPerron}.

\begin{figure}[h!]
	\begin{center}
		\begin{tikzpicture}
		\tikzset{enclosed/.style={draw, circle, inner sep=0pt, minimum size=.10cm, fill=black}}
		\node[enclosed, label={below, yshift=0cm: $d$}] (v_1) at (0,1) {};
		\node[enclosed, label={below, yshift=0cm: $d-1$}] (v_2) at (1,1) {};
		\node[enclosed, label={below, yshift=0cm: $2$}] (v_3) at (2,1) {};
		\node[enclosed, label={below, xshift=-.1cm: $1$}] (v_4) at (3,1) {};
		\node[enclosed,  label={below, yshift=0cm: $d+1$}] (v_6) at (3.5,0.134) {};
		\node[enclosed] (v_7) at (3.866,0.5) {};
		\node[enclosed,label={above, yshift=0cm: $d+r$}] (v_8) at (3.5,1.8660) {};
		\node[enclosed] (v_9) at (3.8660,1.5) {};
		\node[label={center, yshift=0cm: $\vdots$}] (v_10) at (4,1) {};
		
		\draw (v_1) -- (v_2);
		\draw[thick, loosely dotted] (v_2) -- (v_3);
		\draw (v_3) -- (v_4);

		\draw (v_4) -- (v_6);
		\draw (v_4) -- (v_7);
		\draw (v_4) -- (v_8);
		\draw (v_4) -- (v_9);

		\path[draw=black] (v_4) circle[radius=0.15];
		\end{tikzpicture}
	\end{center}
	\caption{A visualization of $B_1(d,r)$ with root \( 1 \).}\label{Figure:Broom2}
\end{figure}
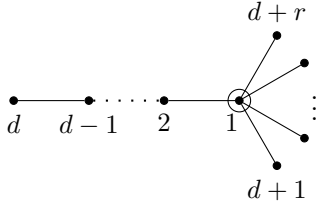

Let $d$ and $r$ be positive integers, and let $n=d+r$. We consider the bottleneck matrix $M$ of the broom $B_1(d,r)$ as in Figure \ref{Figure:Broom2}. Then $M$ is given by
\begin{align*}
	M=\begin{bmatrix}
		M_{P_d} & J_{d,r}\\  
		J_{r,d} & I_r+J_r
	\end{bmatrix}.
\end{align*}
 Let $m_i^{(0)}=1$ for $i=1,\dots,n$. Suppose that for $k\geq 1$, 
$$M\begin{bmatrix}
	m_1^{(k-1)}\\ \vdots \\ m_n^{(k-1)}
\end{bmatrix}=\begin{bmatrix}
	m_1^{(k)}\\ \vdots \\ m_n^{(k)}
\end{bmatrix}.$$
 Then we have
\begin{align*}
	m_{l}^{(k)}=\begin{cases*}
		rm_{d+1}^{(k-1)}+\sum\limits_{i=1}^l\sum\limits_{j=i}^d m_{j}^{(k-1)}& \text{if $1\leq l\leq d$,}\\
		m_{1}^{(k)}+m_{d+1}^{(k-1)}& \text{if $d+1\leq l\leq n$.} 
	\end{cases*}
\end{align*}
Using MATLAB\textsuperscript\textregistered, we obtain
\begin{align*}
	\mathbf{1}^TM^3\mathbf{1}=&r^4 + (4d + 3)r^3 +\frac{1}{2}(2d+3)(d+2)(d+1)r^2\\
	&+\frac{1}{15}(d+1)(4d^4 + 16d^3 + 19d^2 + 21d + 15)r\\
	&+\frac{1}{2520}d(2d+1)(d+1)(68d^4 + 136d^3 + 133d^2 + 65d + 18),\\
	\mathbf{1}^TM^2\mathbf{1}=&r^3 + (3d + 2)r^2 +\frac{1}{3}(d+1)(2d^2+4d+3)r+\frac{1}{30}d(2d+1)(d+1)(2d^2+2d+1).
\end{align*}

\begin{figure}[h!]
	\begin{center}
		\begin{tikzpicture}
		\tikzset{enclosed/.style={draw, circle, inner sep=0pt, minimum size=.10cm, fill=black}}
		\node[enclosed, label={below, yshift=0cm: $1$}] (v_1) at (0,1) {};
		\node[enclosed, label={below, yshift=0cm: $2$}] (v_2) at (1,1) {};
		\node[enclosed, label={below, yshift=0cm: $d-1$}] (v_3) at (2,1) {};
		\node[enclosed, label={below, xshift=-.1cm: $d$}] (v_4) at (3,1) {};
		\node[enclosed,  label={below, yshift=0cm: $d+1$}] (v_6) at (3.5,0.134) {};
		\node[enclosed] (v_7) at (3.866,0.5) {};
		\node[enclosed,label={above, yshift=0cm: $d+r$}] (v_8) at (3.5,1.8660) {};
		\node[enclosed] (v_9) at (3.8660,1.5) {};
		\node[label={center, yshift=0cm: $\vdots$}] (v_10) at (4,1) {};
		
		\draw (v_1) -- (v_2);
		\draw[thick, loosely dotted] (v_2) -- (v_3);
		\draw (v_3) -- (v_4);

		\draw (v_4) -- (v_6);
		\draw (v_4) -- (v_7);
		\draw (v_4) -- (v_8);
		\draw (v_4) -- (v_9);

		\path[draw=black] (v_4) circle[radius=0.15];
		\end{tikzpicture}
	\end{center}
	\caption{A visualization of $B_1(d,r)$ with root \( d \).}\label{Figure:Broom3}
\end{figure}
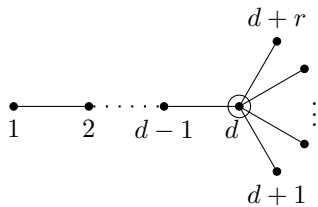

We now consider the neckbottle matrix $Q$ of $B_1(d,r)$ with the labeling of vertices as in Figure \ref{Figure:Broom3}. The path matrix $N$ of $B_1(d,r)$ is given by
\begin{align*}
	N=\begin{bmatrix}
		N_{11} & \be_d\mathbf{1}_r^T\\
		\mathbf{O} & I_r
	\end{bmatrix},
\end{align*}
where $\be_d$ is the column vector of size $d$ with a single $1$ in the $d^\text{th}$ position and zeros elsewhere, and 
$$N_{11}=\begin{bmatrix}
	 1 &  &  \mathbf{O}   \\
	 \vdots &  \ddots & \\
	 1 &\cdots  &  1  \\
\end{bmatrix}.$$
Then $Q$ is given by
\begin{align*}
	Q=\begin{bmatrix}
		M_{P_d}+r\be_d\be_d^T & \be_d\mathbf{1}_r^T\\
		\mathbf{1}_r\be_d^T & I_r
	\end{bmatrix}.
\end{align*}
Let $q_i^{(0)}=1$ for $i=1,\dots,n$. Suppose that for $k\geq 1$, $$Q\begin{bmatrix}
	q_1^{(k-1)}\\ \vdots \\ q_n^{(k-1)}
\end{bmatrix}=\begin{bmatrix}
	q_1^{(k)}\\ \vdots \\ q_n^{(k)}
\end{bmatrix}.$$
From the structure of $Q$, the following recurrence relations can be found:
\begin{align*}
	q_{l}^{(k)}=\begin{cases*}
		\sum\limits_{i=1}^l\sum\limits_{j=i}^d q_{j}^{(k-1)}& \text{if $1\leq l\leq d-1$,}\\
		r\left(q_{d}^{(k-1)}+q_{d+1}^{(k-1)}\right)+\sum\limits_{i=1}^d\sum\limits_{j=i}^d q_{j}^{(k-1)}& \text{if $l=d$,}\\
		q_{d}^{(k-1)}+q_{d+1}^{(k-1)}& \text{if $d+1\leq l\leq n$.}
	\end{cases*}
\end{align*}
Using MATLAB\textsuperscript\textregistered, we obtain
\begin{align*}
	\mathbf{1}^TQ^3\mathbf{1}=&4r^3 + 2(d^2 + 3d + 4)r^2 +\frac{1}{12}(13d^4 + 26d^3 + 41d^2 + 28d + 48)r\\
	&+\frac{1}{2520}d(d+1)(2d+1)(68d^4 + 136d^3 + 133d^2 + 65d + 18),\\
	\mathbf{1}^TQ^2\mathbf{1}=&4r^2 + 2(d^2 + d + 2)r +\frac{1}{30}d(d+1)(2d+1)(2d^2 + 2d + 1).
\end{align*}

We shall show that given $d\geq 3$, $c_3(M,\mathbf{1})-c_3(Q,\mathbf{1})$ has exactly one root in $(0,\infty)$. Using MATLAB\textsuperscript\textregistered, we have
\begin{align*}
	F(d,r)=&\frac{1}{dr}(\mathbf{1}^TM^2\mathbf{1})(\mathbf{1}^TQ^2\mathbf{1})\left(c_3(M,\mathbf{1})-c_3(Q,\mathbf{1})\right) \\
	=&\frac{1}{dr}\left((\mathbf{1}^TM^3\mathbf{1})(\mathbf{1}^TQ^2\mathbf{1})-(\mathbf{1}^TQ^3\mathbf{1})(\mathbf{1}^TM^2\mathbf{1})\right)\\
	=&\frac{1}{60}(d-1)(d-2)(8d^2 - 21d + 11)r^3-\frac{1}{2520}(136d^6 - 868d^5 + 1540d^4 - 1015d^3\\
	& + 2674d^2 - 4417d - 570)r^2-\frac{1}{840}(d-1)(d+1)(24d^5 - 16d^4 - 242d^3 + 355d^2 \\
	&- 116d - 184)r-\frac{1}{37800}(d+1)(2d+1)(8d^7 + 58d^6 + 284d^5 - 575d^4 - 3763d^3\\
	& + 2677d^2 + 4641d + 2970).
\end{align*}
(Note that $F(1,r)=r^2-1$ and $F(2,r)=\frac{1}{2}r^2+\frac{3}{2}r+1$.) Let $d\geq 3$. It can be checked that the coefficients of $r$ and $r^2$ and the constant in $F(d,r)$ are negative. It follows that $F(d,r)$ has only one positive root.
Note that \( F(d,r)=0 \) if and only if \( c_3(M,\mathbf{1})=c_3(Q,\mathbf{1}) \) since \( \frac{1}{dr}(\mathbf{1}^TM^2\mathbf{1})(\mathbf{1}^TQ^2\mathbf{1})>0 \) for \( d,r\ge1 \). Plugging \( r=0.4d^2-1 \) and \( r=0.42d^2+2 \) into \( F(d,r) \), we obtain
\begin{align*}
	F(d,0.4d^2-1)
	=&-\frac{1}{945000}d(96d^9 + 4480d^8 + 17660d^7 - 193050d^6 + 435414d^5 +\\
	& 129255d^4 - 2490255d^3 + 1902425d^2 + 3428125d - 2629350),
\end{align*}
and
\begin{align*}
	F(d,0.42d^2+2)
	=&\frac{1}{472500000}(169344d^{10} - 3415835d^9 + 27443570d^8 - 85370100d^7\\& + 338305446d^6 - 912780225d^5 + 2072349300d^4 - 3997675000d^3 \\&+ 4998385000d^2 - 1712137500d + 1569375000).
\end{align*}
It is easy to see that $F(d,0.4d^2-1)<0$ and $F(d,0.42d^2+2)>0$. By the intermediate value theorem, there exists a real number $r_0$ in the open interval \( (0.4d^2-1,0.42d^2+2) \) such that \( c_3(M,\mathbf{1})=c_3(Q,\mathbf{1}) \). Furthermore, $c_3(M,\mathbf{1})<c_3(Q,\mathbf{1})$ for $r<r_0$ and $c_3(M,\mathbf{1})>c_3(Q,\mathbf{1})$ for $r>r_0$.

\section*{Acknowledgements}

The authors are grateful to Steve Kirkland at the University of Manitoba and Donghyun Kim at Sungkyunkwan University for their valuable comments.

\end{document}